\documentclass{amsart}

\usepackage{amsmath,amssymb,amsthm,amsfonts}
\usepackage{hyperref}

\usepackage{color}

\newtheorem{lemma}{Lemma}[section]
\newtheorem{theorem}{Theorem}[section]

\newtheorem{proposition}{Proposition}[section]
\newtheorem{remark}{Remark}[section]
\newtheorem{corollary}{Corollary}[section]

\numberwithin{equation}{section}

\newcommand{\dis}{\displaystyle}

\newcommand{\C}{\mathbb{C}}
\newcommand{\R}{\mathbb{R}}

\newcommand{\FI}{\mathbf{I}}

\newcommand{\FO}{\mathbf{O}}

\newcommand{\CD}{\mathcal{D}}
\newcommand{\CE}{\mathcal{E}}

\newcommand{\FR}{\mathfrak{R}}

\newcommand{\na}{\nabla}

\newcommand{\al}{\alpha}
\newcommand{\be}{\beta}
\newcommand{\ga}{\gamma}
\newcommand{\om}{\omega}
\newcommand{\la}{\lambda}
\newcommand{\de}{\delta}
\newcommand{\si}{\sigma}
\newcommand{\pa}{\partial}

\newcommand{\eps}{\epsilon}

\newcommand{\De}{\Delta}

\newcommand{\lng}{{\langle}}
\newcommand{\rng}{{\rangle}}

\begin{document}

%
%

\title[The Compressible Euler-Maxwell System: Relaxation Case]{GLOBAL SMOOTH FLOWS FOR THE COMPRESSIBLE EULER-MAXWELL SYSTEM: RELAXATION CASE}

\author[R.-J. Duan]{Renjun Duan}
\address{ (RJD)
Department of Mathematics, The Chinese University of Hong Kong, Shatin\\
Hong Kong}
\email{rjduan@math.cuhk.edu.hk}



\maketitle


\begin{abstract}
The Euler-Maxwell system as a hydrodynamic model for plasma physics
to describe the dynamics of the compressible electrons  in a
constant charged non-moving ion background is studied. The global
smooth flow with small amplitude is constructed in three space
dimensions when the electron velocity relaxation is present. The
speed of the electrons flow trending to uniform equilibrium is
obtained. The pointwise behavior of solutions to the linearized homogeneous
system in the frequency space is also investigated in detail.
\end{abstract}

\medskip

Keywords: Euler-Maxwell system; global existence; large time behavior.

\section{Introduction}
The Euler-Maxwell system is a hydrodynamic model in plasma physics
to describe the dynamics of electrons and ions under the influence
of their self-consistent electromagnetic field \cite{RG,MRS}. Starting from the Euler-Maxwell system, some hierarchies of models such as the Dynamo hierarchy and the MHD hierarchy can be derived under the different situations about the state of the plasma \cite{BCD}. The Euler-Maxwell system in some cases can also be justified as the asymptotic limit of the kinetic Vlasov-Maxwell system by the so-called quasi-neutral regime \cite{BMP}.  In a
simple case when the constant positive charged ions do not move
providing only a uniform background and the electrons flow is
isentropic, the compressible Euler-Maxwell system takes the form of
\begin{equation}\label{s.o}
\left\{\begin{array}{l}
  \dis \pa_t n+\na \cdot (n u)=0,\\[3mm]
  \dis \pa_t u+ u\cdot \na u +\frac{1}{n}\na p(n)=-(E+u\times B)-\nu u,\\[3mm]
  \dis \pa_t E-\na \times B= nu,\\[3mm]
  \dis \pa_t B+\na \times E=0,\\[3mm]
  \dis \na\cdot E=n_{\rm b}-n,\ \ \na\cdot B=0.
\end{array}\right.
\end{equation}
Here, $n=n(t,x)\geq 0$, $u=u(t,x)\in \R^3$, $E=E(t,x)\in \R^3$ and $B=B(t,x)\in\R^3$, for $t>0$, $x\in \R^3$, denote the electron density, electron velocity, electric field and magnetic field, respectively. Initial data is given as
\begin{equation}\label{s.o.id}
    [n,u,E,B]|_{t=0}=[n_0,u_0,E_0,B_0],\ \ x\in \R^3,
\end{equation}
with the compatible condition
\begin{equation}\label{s.o.cc}
    \na\cdot E_0=n_{\rm b}-n_0,\ \ \na\cdot B_0=0,\ \ x\in \R^3.
\end{equation}
The pressure function $p(\cdot)$ of the flow depending only on the density  satisfies the  power law $p(n)=An^\ga$ with constants $A>0$ and $\ga>1$, where $\ga$ is the adiabatic exponent. Constants $\nu>0$ and $n_{\rm b}>0$ are the velocity relaxation frequency and the equilibrium-charged density of ions, respectively. Through this paper, we set $A=1$, $\nu=1$ and $n_{\rm b}=1$ without loss of generality. In addition, the case of $\ga=1$ can be considered in the same way.

There are some mathematical studies on the above Euler-Maxwell system. By using the fractional Godunov scheme as well as the compensated compactness argument, Chen-Jerome-Wang \cite{CJW} proved global existence of weak solutions to the initial-boundary value problem in one space dimension for arbitrarily large initial data in $L^\infty$. Jerome \cite{Je} provided a local smooth solution theory for the Cauchy problem over $\R^3$ by adapting the classical semigroup-resolvent approach of Kato \cite{Kato}. Peng-Wang \cite{PW} established convergence of the compressible Euler-Maxwell system to the incompressible Euler system for well-prepared smooth initial data. Much more studies have been made for the Euler-Poisson system when the magnetic field is absent; see \cite{Guo,LNX,DLYY,LS,CT} and references therein for discussion and analysis of the different issues such as the existence of global smooth irrotational flow \cite{Guo}, large time behavior of solutions \cite{LNX}, stability of star solutions \cite{DLYY,LS} and finite time blow-up \cite{CT}.

On the other hand, the existence and uniqueness of global solutions to the Euler-Maxwell system in three space dimensions remains an open problem. In this paper, we answer it in the framework of smooth solutions with small amplitude. The main result is stated as follows.

\begin{theorem}\label{thm.s.o}
Let $N\geq 4$ and \eqref{s.o.cc} hold. There are  $\de_0>0$, $C_0$ such that if
\begin{equation*}
   \|[n_0-1,u_0,E_0,B_0]\|_N\leq \de_0,
\end{equation*}
then, the Cauchy problem \eqref{s.o}-\eqref{s.o.id} admits a unique global solution $[n(t,x),\linebreak u(t,x),E(t,x),B(t,x)]$ with
\begin{equation*}
    [n-1,u,E,B] \in C([0,\infty);H^{N}(\R^3))\cap Lip([0,\infty);H^{N-1}(\R^3))
\end{equation*}
and
\begin{equation*}
    \sup_{t\geq 0}\|[n(t)-1,u(t),E(t),B(t)]\|_N\leq C_0 \|[n_0-1,u_0,E_0,B_0]\|_N.
\end{equation*}
Moreover, there are $\de_1>0$, $C_1$ such that if
\begin{equation*}
   \|[n_0-1,u_0,E_0,B_0]\|_{13}+\| [u_0,E_0,B_0]\|_{L^1}\leq \de_1,
\end{equation*}
then the solution $[n(t,x),u(t,x),E(t,x),B(t,x)]$ satisfies that for any $t\geq 0$,
\begin{eqnarray*}
 && \|n(t)-1\|_{L^q}\leq C_1 (1+t)^{-\frac{11}{4}},\\
 &&\|[u(t),E(t)]\|_{L^q}\leq C_1 (1+t)^{-2+\frac{3}{2q}},\\
 &&      \|B(t)\|_{L^q}\leq C_1 (1+t)^{-\frac{3}{2}+\frac{3}{2q}},
\end{eqnarray*}
with $2\leq q\leq \infty$.
\end{theorem}

It is obvious that when $N$ is large enough, the solution is classical belonging to $C^1([0,\infty)\times \R^3)$ and particularly when initial perturbation is smooth, the solution is also smooth. Here we remark that the Euler-Maxwell system in the whole space $\R^3$ is dispersive. Notice that the usual homogeneous Maxwell system for the electromagnetic field conserves the energy. But when the electromagnetic field is generated by the compressible electron flow, it will show a weak dispersive property and thus decay in time with some algebraic rate, which is essentially due to the coupling of the Maxwell system with the Euler equations. Furthermore, the weak dispersive property of the Maxwell system also leads to the fact that the time-decay speed of the magnetic field $B$ is the slowest among all the components of the solution. Finally, it should be emphasized that the velocity relaxation term of the considered Euler-Maxwell system here plays a key role in the proof of Theorem \ref{thm.s.o}. We shall study in the other forthcoming work the case of non-relaxation for which the proof is much more complicated to carry out.

Let us introduce some notations for the use throughout this paper. $C$  denotes
some positive (generally large) constant and $\la$ denotes some positive (generally small) constant, where both $C$ and
$\la$ may take different values in different places. For two quantities $a$ and $b$,
$a\sim b$ means $\la a \leq b \leq \frac{1}{\la} a $ for a generic constant $0<\la<1$. For any
integer $m\geq 0$, we use $H^m$, $\dot{H}^m$ to denote the
usual Sobolev space $H^m(\R^3)$ and the corresponding $m$-order homogeneous Sobolev space, respectively. Set $L^2=H^m$ when $m=0$. For simplicity, the norm of $H^m$ is denoted by $\|\cdot\|_m$ with $\|\cdot\|=\|\cdot\|_0$.
We use $\langle\cdot,\cdot\rangle$ to denote the inner product over
the Hilbert space $L^2(\R^3)$, i.e.
\begin{equation*}
    \langle f,g\rangle=\int_{\R^3} f(x)g(x)dx,\ \ f=f(x),g=g(x)\in
    L^2(\R^3).
\end{equation*}
For a multi-index $\al=[\al_1,\al_2,\al_3]$, we denote
$\pa^{\al}=\pa_{x_1}^{\al_1}\pa_{x_2}^{\al_2}\pa_{x_3}^{\al_3}$.
The length of $\al$ is $|\al|=\al_1+\al_2+\al_3$. For simplicity,
we also set $\pa_j=\pa_{x_j}$ for $j=1,2,3$.

We conclude this section by stating the arrangement of the rest of this paper. In Section \ref{sec.reform}, we reformulate the Cauchy problem under consideration. In Section \ref{sec.exi}, we prove the global existence and uniqueness of solutions. In Section \ref{sec.lhs}, we investigate the linearized homogeneous system to obtain the $L^p$-$L^q$ time-decay property and the explicit representation of solutions. Finally, in Section \ref{sec.nonldecay}, we study the time-decay rates of solutions to the reformulated nonlinear system and finish the proof of Theorem \ref{thm.s.o}.

\medskip

\section{Reformulation of the problem}\label{sec.reform}

Let $[n,u,E,B]$ be a smooth solution to the Cauchy problem of the Euler-Maxwell system \eqref{s.o} with given initial data \eqref{s.o.id} satisfying \eqref{s.o.cc}. Set
\begin{equation}\label{trans.1}
   \left\{ \begin{array}{c}
      \dis \si(t,x)=\frac{2}{\ga-1}\{[n(\frac{t}{\sqrt{\ga}},x)]^{\frac{\ga-1}{2}}-1\},\ \
      \dis v(t,x)=\frac{1}{\sqrt{\ga}}u(\frac{t}{\sqrt{\ga}},x),\\[5mm] \dis\widetilde{E}(t,x)=\frac{1}{\sqrt{\ga}}E(\frac{t}{\sqrt{\ga}},x),\ \
      \widetilde{B}(t,x)=\frac{1}{\sqrt{\ga}}B(\frac{t}{\sqrt{\ga}},x).
    \end{array}\right.
\end{equation}
Then, $V:=[\si,v,\widetilde{E},\widetilde{B}]$ satisfies
\begin{equation}\label{s.e}
\left\{\begin{array}{l}
  \dis \pa_t \si+v\cdot \na \si+(\frac{\ga-1}{2}\si+1)\na\cdot v=0,\\[3mm]
  \dis \pa_t v+ v\cdot \na v +(\frac{\ga-1}{2}\si+1)\na \si=-(\frac{1}{\sqrt{\ga}}\widetilde{E}+v\times \widetilde{B})-\frac{1}{\sqrt{\ga}}v,\\[3mm]
  \dis \pa_t \widetilde{E}-\frac{1}{\sqrt{\ga}}\na \times \widetilde{B}=\frac{1}{\sqrt{\ga}}v+\frac{1}{\sqrt{\ga}}[\si+\Phi(\si)]v,\\[3mm]
  \dis \pa_t \widetilde{B}+\frac{1}{\sqrt{\ga}}\na \times \widetilde{E}=0,\\[3mm]
  \dis \na\cdot \widetilde{E}=-\frac{1}{\sqrt{\ga}}[\si+\Phi(\si)],\ \ \na\cdot \widetilde{B}=0,\ \ t>0,x\in \R^3,
\end{array}\right.
\end{equation}
with initial data
\begin{equation}\label{s.e.id}
    V|_{t=0}=V_0:=[\si_0,v_0,\widetilde{E}_0,\widetilde{B}_0],\ \ x\in \R^3.
\end{equation}
Here, $ \Phi(\cdot)$ is defined by
\begin{equation*}
    \Phi(\si)=(\frac{\ga-1}{2}\si+1)^{\frac{2}{\ga-1}}-\si-1,
\end{equation*}
and $V_0=[\si_0,v_0,\widetilde{E}_0,\widetilde{B}_0]$ is given from $[n_0,u_0,E_0,B_0]$ according to the transform \eqref{trans.1}, and hence $V_0$ satisfies
\begin{equation}\label{s.e.cc}
    \na\cdot \widetilde{E}_0=-\frac{1}{\sqrt{\ga}}[\si_0+\Phi(\si_0)],\ \ \na\cdot \widetilde{B}_0=0, \ \ x\in \R^3.
\end{equation}
In the rest of this paper, to prove Theorem \ref{thm.s.o}, we are reduced to mainly investigate the well-posedness and large-time behavior for solutions to the reformulated Cauchy problem \eqref{s.e}-\eqref{s.e.id} with the compatible condition \eqref{s.e.cc}. In addition, when the large-time behavior of  solutions is considered, it is more convenient to use another reformulation of the original  Cauchy problem \eqref{s.o}-\eqref{s.o.id}. In fact, by setting $\rho(t,x)=n(t,x)-1$, then $U:=[\rho,u,E,B]$ satisfies
\begin{equation}\label{s.cr}
\left\{\begin{array}{l}
  \dis \pa_t \rho+\na \cdot u=-\na\cdot (\rho u),\\[3mm]
  \dis \pa_t u+ \ga \na \rho+E+u=-u\cdot \na u -u\times B-\ga[(1+\rho)^{\ga-2}-1]\na \rho,\\[3mm]
  \dis \pa_t E-\na \times B-u= \rho u,\\[3mm]
  \dis \pa_t B+\na \times E=0,\\[3mm]
  \dis \na\cdot E=-\rho,\ \ \na\cdot B=0,\ \ t>0,x\in \R^3,
\end{array}\right.
\end{equation}
with initial data
\begin{equation}\label{s.cr.id}
   U|_{t=0}=U_0:=[\rho_0,u_0,E_0,B_0],\ \ x\in \R^3,
\end{equation}
satisfying
\begin{equation}\label{s.cr.cc}
    \na\cdot E_0=-\rho_0,\ \ \na\cdot B_0=0.
\end{equation}
Here, $\rho_0=n_0-1$.

In what follows, we suppose the integer $N\geq 4$. Besides, for $V=[\si,v,\widetilde{E},\widetilde{B}]$, we define the full instant energy functional $\CE_N(V(t))$, the high-order instant energy functional $\CE_N^{\rm h}(V(t))$, and the corresponding dissipation rates $\CD_N(V(t))$, $\CD_N^{\rm h}(V(t))$ by
\begin{eqnarray}
  \CE_N(V(t)) &\sim& \|[\si,v,\widetilde{E},\widetilde{B}]\|_N^2,\label{def.ef}\\
 \CE_N^{\rm h}(V(t)) &\sim& \|\na [\si,v,\widetilde{E},\widetilde{B}]\|_{N-1}^2,\label{def.ef.h}
\end{eqnarray}
and
\begin{eqnarray}
  \CD_N(V(t)) &=& \|[\si, v]\|_N^2+\|\na[\widetilde{E},\widetilde{B}]\|_{N-2}^2+\|\widetilde{E}\|^2,\label{def.ef.d}\\
  \CD_N^{\rm h}(V(t)) &=& \|\na [\si, v]\|_{N-1}^2+\|\na[\widetilde{E},\widetilde{B}]\|_{N-2}^2.\label{def.ef.h.d}
\end{eqnarray}
Then, concerning the reformulated Cauchy problem \eqref{s.e}-\eqref{s.e.id}, one has the following global existence result.

\begin{proposition}\label{prop.s.exi}
Suppose \eqref{s.e.cc} for given initial data $V_0=[\si_0,v_0,\widetilde{E}_0,\widetilde{B}_0]$. Then, there are $\CE_N(\cdot)$ and $\CD_N(\cdot)$ in the form of \eqref{def.ef} and \eqref{def.ef.d}  such that the following holds true. If $\CE_N(V_0)>0$ is sufficiently small, the Cauchy problem \eqref{s.e}-\eqref{s.e.id} admits a unique global  nonzero solution $V=[\si,v,\widetilde{E},\widetilde{B}]$ satisfying
\begin{equation}\label{prop.s.exi.1}
   V \in C([0,\infty);H^{N}(\R^3))\cap Lip([0,\infty);H^{N-1}(\R^3)),
\end{equation}
and
\begin{equation}\label{prop.s.exi.2}
   \CE_N(V(t)) +\la \int_0^t\CD_N(V(s))ds\leq \CE_N(V_0)
\end{equation}
for any $t\geq0$.
\end{proposition}

\begin{remark}
From \eqref{prop.s.exi.2} and \eqref{def.ef.d}, $\si$, $v$ and $\widetilde{E}$ are time-space integrable but $\widetilde{B}$ is not so. For the derivatives, $[\si,v]$ is time-space integrable up to $N$-order but $[\widetilde{E},\widetilde{B}]$ is so up to $N-1$ order only. Therefore, the Euler-Maxwell system is not only degenerately dissipative but also of the regularity-loss type. The similar phenomenon has been noticed in \cite{DS-VMB} for the study of the optimal large-time behavior of solutions to the two-species Vlasov-Maxwell-Boltzmann system.

\end{remark}

Moreover, solutions obtained in Proposition \ref{prop.s.exi} indeed decay in time with some rates under some extra regularity and integrability conditions on initial data. For that, given $V_0=[\si_0,v_0,\widetilde{E}_0,\widetilde{B}_0]$, set $\eps_m(V_0)$ as
\begin{equation}\label{def.eps.id}
\eps_m(V_0)=\|V_0\|_{m}+\|[v_0,\widetilde{E}_0,\widetilde{B}_0]\|_{L^1},
\end{equation}
for the integer $m\geq 4$. Then, one has the following two propositions.

\begin{proposition}\label{prop.s.decay}
Suppose that $V_0=[\si_0,v_0,\widetilde{E}_0,\widetilde{B}_0]$ satisfies \eqref{s.e.cc}. If $\eps_{N+2}(V_0)>0$ is sufficiently small, then  the solution $V=[\si,v,\widetilde{E},\widetilde{B}]$ satisfies
\begin{equation}\label{prop.s.decay.1}
      \|V(t)\|_{N}\leq C \eps_{N+2}(V_0)(1+t)^{-\frac{3}{4}}
\end{equation}
for any $t\geq 0$. Furthermore, if $\eps_{N+6}(V_0)>0$ is sufficiently small, then  the solution $V=[\si,v,\widetilde{E},\widetilde{B}]$ also satisfies
\begin{equation}\label{prop.s.decay.2}
      \|\na V(t)\|_{N-1}\leq C \eps_{N+6}(V_0)(1+t)^{-\frac{5}{4}}
\end{equation}
for any $t\geq 0$.
\end{proposition}

\begin{proposition}\label{prop.s.Lq}
Let $2\leq q\leq \infty$. Suppose that  $V_0=[\si_0,v_0,\widetilde{E}_0,\widetilde{B}_0]$ satisfies\eqref{s.e.cc} and $\eps_{13}(V_0)>0$ is sufficiently small. Then, the solution $V=[\si,v,\widetilde{E},\widetilde{B}]$ satisfies that for any $t\geq 0$,
\begin{eqnarray}
 && \|\si(t)\|_{L^q}\leq C (1+t)^{-\frac{11}{4}},\label{prop.s.Lq.1}\\
 &&\|[v(t),\widetilde{E}(t)]\|_{L^q}\leq C (1+t)^{-2+\frac{3}{2q}},\label{prop.s.Lq.2}\\
 &&      \|\widetilde{B}(t)\|_{L^q}\leq C (1+t)^{-\frac{3}{2}+\frac{3}{2q}}.\label{prop.s.Lq.3}
\end{eqnarray}
\end{proposition}

\begin{remark}
Proposition \ref{prop.s.decay} shows that for the slower time-decay rate described by \eqref{prop.s.decay.1}, initial data needs the extra $H^2$ space regularity, while for the faster decay rate as in \eqref{prop.s.decay.2}, initial data needs the extra $H^6$ space regularity. The regularity index $13$ from $\eps_{13}(V_0)>0$ in Proposition \ref{prop.s.Lq} comes out due to Proposition \ref{prop.s.decay} and the bootstrap argument. Notice that in terms of the definition \eqref{def.eps.id} of $\eps_{m}(V_0)>0$, we do not suppose that $\|\si_0\|_{L^1}$ is sufficiently small in both Proposition \ref{prop.s.decay} and Proposition \ref{prop.s.Lq}. This is non-trivial on the basis of the analysis of the time-decay property of solutions to the linearized homogeneous system; see Theorem \ref{thm.lhs.re} and Corollary \ref{cor.lhs.re}.

\end{remark}

Finally, it is easy to see that Theorem \ref{thm.s.o} follows from Proposition \ref{prop.s.exi} and Proposition \ref{prop.s.Lq}. Thus, the rest of this paper is to prove the stated-above three propositions.


\section{Global solutions for the nonlinear system}\label{sec.exi}

In this section, we shall prove Proposition \ref{prop.s.exi} for the global existence and uniqueness of solutions to the Cauchy problem  \eqref{s.e}-\eqref{s.e.id}.  In the first subsection, we obtain some uniform-in-time a priori estimates for any smooth solution. In the second subsection, we combine those a priori estimates with the local existence of solutions to extend the local solution up to infinite time with the help of the continuity argument.

\subsection{A priori estimates}

We begin to use the normal energy method to obtain some uniform-in-time a priori estimates for smooth solutions to the Cauchy problem  \eqref{s.e}-\eqref{s.e.id}. Notice that \eqref{s.e} is a quasi-linear symmetric hyperbolic system. The main goal of this subsection is to prove

\begin{theorem}[a priori estimates]\label{thm.ap}
Suppose
$$
V=[\si,v,\widetilde{E},\widetilde{B}]\in C([0,T);H^{N}(\R^3))
$$
is smooth for $T>0$ with
\begin{equation}\label{thm.ap.1}
    \sup_{0\leq t<T}\|\si(t)\|_{N}\leq 1,
\end{equation}
and assume that $V$ solves the system \eqref{s.e} for $t\in(0,T)$.
Then, there are $\CE_N(\cdot)$ and $\CD_N(\cdot)$ in the form of \eqref{def.ef} and \eqref{def.ef.d} such that
\begin{equation}\label{thm.ap.2}
    \frac{d}{dt}\CE_N(V(t))+\la \CD_N(V(t))\leq C \left[\CE_N(V(t))^{1/2}+\CE_N(V(t))\right]\CD_N(V(t))
\end{equation}
for any $0\leq t<T$.
\end{theorem}

\begin{proof}
It is divided by five steps as follows.

\medskip

\noindent{\it Step 1.} It holds that
\begin{equation}\label{thm.ap.p01}
    \frac{1}{2}\frac{d}{dt}\|V\|_N^2+\frac{1}{\sqrt{\ga}}\|v\|_N^2\leq C\|V\|_N(\|v\|^2+\|\na[\si,v]\|_{N-1}^2).
\end{equation}
In fact, from the first two equations of \eqref{s.e}, energy estimates on $\pa^\al\si$ and $\pa^\al v$ for $|\al|\leq N$ give
\begin{equation}\label{thm.ap.p02}
    \frac{1}{2}\frac{d}{dt}\|\pa^\al [\si,v]\|^2+\frac{1}{\sqrt{\ga}}\|\pa^\al v\|^2+\frac{1}{\sqrt{\ga}}\langle\pa^\al \widetilde{E}, \pa^\al v\rangle=-\sum_{\be<\al}C^{\al}_{\be}I_{\al,\be}(t)+I_1(t),
\end{equation}
with
\begin{eqnarray*}
  I_{\al,\be}(t) &=& \langle\pa^{\al-\be} v\cdot \na \pa^\be \si,\pa^\al \si\rangle
  +\langle\pa^{\al-\be} v\cdot \na \pa^\be v,\pa^\al v\rangle\\
  &&+\frac{\ga+1}{2}\langle\pa^{\al-\be} \si\na\cdot  \pa^\be v,\pa^\al \si\rangle
  +\frac{\ga+1}{2}\langle\pa^{\al-\be} \si\na\pa^\be v,\pa^\al \si\rangle\\
  &&+\langle\pa^{\al-\be} v\times \pa^\be \widetilde{B},\pa^\al v\rangle
\end{eqnarray*}
and
\begin{equation*}
  I_1(t) = \frac{1}{2}\langle \na\cdot v, |\pa^\al \si|^2+|\pa^\al v|^2\rangle+\frac{\ga+1}{2}\langle\na\si\cdot \pa^\al v,\pa^\al \si\rangle-\langle v\times \widetilde{B}, \pa^\al v\rangle,
\end{equation*}
where integration by parts were used. When $|\al|=0$, it suffices to estimate $I_1(t)$ by
\begin{eqnarray*}
  I_1(t) &\leq & C\|\na\cdot v\|_{L^2}(\|\si\|_{L^6}\|\si\|_{L^3}+\|v\|_{L^6}\|v\|_{L^3})\\
  &&+C\|\na\si\|_{L^2}\|\si\|_{L^6}\|v\|_{L^2}+C\|\widetilde{B}\|_{L^\infty}\|v\|_{L^2}^2\\
  &\leq &C\|[\si,v]\|_{H^1}\|\na [\si,v]\|^2+C\|\na \widetilde{B}\|_{H^1}\|v\|^2,
\end{eqnarray*}
which is further bounded by the r.h.s. term of \eqref{thm.ap.p01}. When $|\al|\geq 1$, since each term in $I_{\al,\be}(t)$ and $I_1(t)$ is the integration of the  three-terms product in which there is at least one term containing the derivative, one has
\begin{equation*}
|I_{\al,\be}(t)|+|I_1(t)|\leq C\|[\si,v,\widetilde{B}]\|_N\|\na [\si,v]\|_{N-1}^2,
\end{equation*}
which is also further bounded by the r.h.s. term of \eqref{thm.ap.p01}. On the other hand, from \eqref{s.e}, energy estimates on $\pa^\al\widetilde{E}$ and $\pa^\al \widetilde{B}$ with $|\al|\leq N$
give
\begin{equation}\label{thm.ap.p02.1}
    \frac{1}{2}\frac{d}{dt}\|\pa^\al [\widetilde{E},\widetilde{B}]\|^2-\frac{1}{\sqrt{\ga}}\langle\pa^\al v,\pa^\al \widetilde{E}\rangle \leq \frac{1}{\sqrt{\ga}}\langle \pa^\al [(\si+\Phi(\si))v],\pa^\al \widetilde{E}\rangle:=I_2(t).
\end{equation}
In a similar way as before, when $|\al|=0$,
\begin{equation*}
    I_2(t)\leq C\|\na \si\|\cdot \|v\|_1\|\widetilde{E}\|,
\end{equation*}
and when $|\al|>0$,
\begin{equation*}
    I_2(t)\leq C\|\na \si\|_{N-1}\|\na v\|_{N-1}\|\na \widetilde{E}\|_{N-1}.
\end{equation*}
Thus, for $|\al|\leq N$, one has
\begin{equation*}
    I_2(t)\leq C\|\widetilde{E}\|_{N}(\|\na [\si,v]\|_{N-1}^2+\|v\|^2),
\end{equation*}
which is bounded by the r.h.s. term of \eqref{thm.ap.p01}. Then, \eqref{thm.ap.p01} follows by taking summation of  \eqref{thm.ap.p02} and \eqref{thm.ap.p02.1} over $|\al|\leq N$. Here, we stop to remark that in this step, the time evolution of the full instant energy $\|V(t)\|_N^2$ has been obtained but its dissipation rate only contains the contribution from the explicit relaxation variable $v$. In the following three steps, by introducing some interactive functionals, the dissipation from contributions of the rest components $\si,\widetilde{E}$ and $\widetilde{B}$ can be recovered in turn.

\medskip

\noindent{\it Step 2.} It holds that
\begin{equation}\label{thm.ap.p03}
   \frac{d}{dt}\CE_{N,1}^{\rm int}(V)+\la \|\si\|_N^2\leq C\|\na v\|_{N-1}^2+ C\|[\si,v,\widetilde{B}]\|_N^2\|\na[\si,v]\|_{N-1}^2,
\end{equation}
where $\CE_{N,1}^{\rm int}(\cdot)$ is defined by
\begin{equation*}
    \CE_{N,1}^{\rm int}(V)=\sum_{|\al|\leq N-1}\lng \pa^\al v,\pa^\al \na \si\rng.
\end{equation*}
In fact, notice that the first two equations of \eqref{s.e} can be rewritten as
\begin{align}
&\pa_t \si +\na\cdot v=f_1,\ \ f_1:=-v\cdot \na\si-\frac{\ga+1}{2}\si \na\cdot v,\label{thm.ap.p05}\\
&\pa_t v+\na \si +\frac{1}{\sqrt{\ga}}\widetilde{E}=-\frac{1}{\sqrt{\ga}} v+f_2,\ \ f_2:=-v\cdot \na v-\frac{\ga+1}{2}\si \na\si-v\times \widetilde{B}.    \label{thm.ap.p06}
\end{align}
Let $|\al|\leq N-1$. Applying $\pa^\al$ to \eqref{thm.ap.p06}, multiplying it by $\pa^\al \na \si$, taking integrations in $x$ and then using integration by parts and also the final equation of \eqref{s.e} gives
\begin{multline*}
 \frac{d}{dt}\lng \pa^\al v,\pa^\al\na \si\rng +\|\pa^\al \na \si\|^2+\frac{1}{\ga}\|\pa^\al\si\|^2=\lng \pa^\al v,\pa^\al \na\pa_t \si\rng\\
  -\frac{1}{\sqrt{\ga}}\lng \pa^\al v,\pa^\al \na \si\rng
-\frac{1}{\ga}\lng \pa^\al \Phi(\si),\pa^\al \si\rng +\lng \pa^\al f_2,\pa^\al \na \si\rng,
\end{multline*}
which further by replacing $\pa_t\si$ from \eqref{thm.ap.p05}, implies
\begin{multline*}
 \frac{d}{dt}\lng \pa^\al v,\pa^\al\na \si\rng +\|\pa^\al \na \si\|^2+\frac{1}{\ga}\|\pa^\al\si\|^2\\
 =\|\pa^\al \na\cdot v\|^2
  -\frac{1}{\sqrt{\ga}}\lng \pa^\al v,\pa^\al \na \si\rng
-\frac{1}{\ga}\lng \pa^\al \Phi(\si),\pa^\al \si\rng\\
-\lng \pa^\al f_1,\pa^\al \na\cdot v\rng +\lng \pa^\al f_2,\pa^\al \na \si\rng.
\end{multline*}
Then, it follows from Cauchy-Schwarz inequality that
\begin{multline}
\label{thm.ap.p07}
 \frac{d}{dt}\lng \pa^\al v,\pa^\al\na \si\rng +\la(\|\pa^\al \na \si\|^2+\|\pa^\al\si\|^2)\\
 \leq C\|\pa^\al \na\cdot v\|^2+
  C(\|\pa^\al \Phi(\si)\|^2+\|\pa^\al f_1\|^2+\|\pa^\al f_2\|^2).
\end{multline}
Noticing that $\Phi(\si)$ is smooth in $\si$ with $\Phi(0)=\Phi'(0)=0$ and $f_1,f_2$ are quadratically nonlinear, one has from \eqref{thm.ap.1} that
\begin{equation*}
   \|\pa^\al \Phi(\si)\|^2+\|\pa^\al f_1\|^2+\|\pa^\al f_2\|^2
   \leq C\|[\si,v,\widetilde{B}]\|_{N}^2\|\na [\si,v]\|_{N-1}^2.
\end{equation*}
Plugging this into \eqref{thm.ap.p07} and taking summation over $|\al|\leq N-1$ yields \eqref{thm.ap.p03}.

\medskip

\noindent{\it Step 3.} It holds that
\begin{eqnarray}
\label{thm.ap.p08}
\frac{d}{dt}\CE_{N,2}^{\rm int}(V)+\la \|\widetilde{E}\|_{N-1}^2&\leq& C\|[\si,v]\|_N^2 +C\|v\|_{N}\|\na \widetilde{B}\|_{N-2}\\
&&+ C\|[\si,v,\widetilde{B}]\|_N^2\|\na[\si,v]\|_{N-1}^2,\nonumber
\end{eqnarray}
where $\CE_{N,2}^{\rm int}(\cdot)$ is defined by
\begin{equation*}
    \CE_{N,2}^{\rm int}(V)=\sum_{|\al|\leq N-1} \lng \pa^\al v,\pa^\al\widetilde{E}\rng.
\end{equation*}
In fact, for $|\al|\leq N-1$, applying $\pa^\al$ to \eqref{thm.ap.p06}, multiplying it by $\pa^\al \widetilde{E}$, taking integration in $x$ and then using the third equation of \eqref{s.e} gives
\begin{multline*}
 \frac{d}{dt}\lng \pa^\al v,\pa^\al\widetilde{E}\rng +\frac{1}{\sqrt{\ga}}\|\pa^\al \widetilde{E}\|^2\\
=\frac{1}{\sqrt{\ga}}\|\pa^\al v\|^2+\frac{1}{\sqrt{\ga}}\lng \pa^\al v,\na\times \pa^\al \widetilde{B} \rng
+\frac{1}{\ga}\lng \pa^\al v, \pa^\al [\si v+\Phi(\si) v]\rng\\
-\lng \na \pa^\al \si+\frac{1}{\sqrt{\ga}}\pa^\al v,\pa^\al \widetilde{E}\rng+\lng \pa^\al f_2, \pa^\al \widetilde{E}\rng,
\end{multline*}
which from Cauchy-Schwarz inequality further implies
\begin{multline*}
\frac{d}{dt}\lng \pa^\al v,\pa^\al\widetilde{E}\rng +\la \|\pa^\al \widetilde{E}\|^2 \leq C\|[\si,v]\|_{N}^2+C\|v\|_{N}\|\na\widetilde{B}\|_{N-2}\\
+ C\|[\si,v,\widetilde{B}]\|_N^2\|\na[\si,v]\|_{N-1}^2.
\end{multline*}
Thus, \eqref{thm.ap.p08} follows from taking summation of the above estimate over $|\al|\leq N-1$.

\medskip

\noindent{\it Step 4.} It holds that
\begin{equation}\label{thm.ap.p09}
   \frac{d}{dt}\CE_{N,3}^{\rm int}(V)+\la \|\na \widetilde{B}\|_{N-2}^2\leq C\|[v,\widetilde{E}]\|_{N-1}^2+ C\|\si\|_N^2\|\na v\|_{N-1}^2,
\end{equation}
where $\CE_{N,3}^{\rm int}(\cdot)$ is defined by
\begin{equation*}
    \CE_{N,3}^{\rm int}(V)= \sum_{|\al|\leq N-2}\lng\na\times \pa^\al \widetilde{E}, \pa^\al \widetilde{B}\rng.
\end{equation*}
In fact, for $|\al|\leq N-2$, applying $\pa^\al$ to the third equation of \eqref{s.e}, multiplying it by $\pa^\al \na\times \widetilde{B}$, taking integration in $x$ and then using the fourth equation of \eqref{s.e} implies
\begin{multline*}
\frac{d}{dt}\lng\na\times \pa^\al \widetilde{E}, \pa^\al \widetilde{B}\rng +\frac{1}{\sqrt{\ga}}\|\na\times \pa^\al\widetilde{B}\|^2\\
=\frac{1}{\sqrt{\ga}}\|\na\times \pa^\al \widetilde{E}\|^2-\frac{1}{\sqrt{\ga}}\lng \pa^\al v,\na\times \pa^\al\widetilde{B}\rng-\frac{1}{\sqrt{\ga}}\lng \pa^\al[\si v+\Phi(\si)v],\na\times  \pa^\al\widetilde{B}\rng,
\end{multline*}
which gives \eqref{thm.ap.p09} by further using Cauchy-Schwarz inequality and taking summation over $|\al|\leq N-2$, where we also used
\begin{equation*}
    \|\pa^\al\pa_i\widetilde{B}\|=\|\pa_i\De^{-1}\na\times (\na\times \pa^\al\widetilde{B})\|\leq C\|\na\times \pa^\al\widetilde{B}\|
\end{equation*}
for each $1\leq i\leq 3$, due to the fact that $\pa_i\De^{-1}\na$ is bounded from $L^p$ to itself for $1<p<\infty$; see \cite{Stein}.

\medskip

\noindent{\it Step 5.} Now, following four steps above, we are ready to  prove \eqref{thm.ap.2}. Here, we first remark that \eqref{thm.ap.p03} implies that the dissipation of $\si$ can be recovered from that of $v$, \eqref{thm.ap.p08} implies that the dissipation of $\widetilde{E}$ can be recovered from that of $v$, $\si$ and $\widetilde{B}$, and \eqref{thm.ap.p09} implies that the dissipation of $\widetilde{B}$ can be recovered from that of $v$ and $\widetilde{E}$. The key observation is that the second term on the r.h.s. of \eqref{thm.ap.p08} is the product of dissipations of $v$ and $\widetilde{B}$ so that it is possible to recover the full dissipation of $v,\si,\widetilde{E}$ and $\widetilde{B}$ by taking a proper linear combination of all estimates. In fact, let us define
\begin{equation*}
    \CE_N(V(t))=\|V(t)\|_N^2+\sum_{i=1}^3\kappa_i \CE_{N,i}^{\rm int}(V(t)),
\end{equation*}
that is,
\begin{eqnarray}
  \CE_N(V(t)) &=& \|[\si,v,\widetilde{E},\widetilde{B}]\|^2_N+\kappa_1\sum_{|\al|\leq N-1}\lng\pa^\al \na \si, \pa^\al v\rng\nonumber\\
  &&+\kappa_2\sum_{|\al|\leq N-1} \lng \pa^\al v,\pa^\al\widetilde{E}\rng
  +\kappa_3\sum_{|\al|\leq N-2}\lng\na\times \pa^\al \widetilde{E}, \pa^\al \widetilde{B}\rng\label{def.energy}
\end{eqnarray}
for constants $0<\kappa_3\ll \kappa_2\ll \kappa_1\ll 1$ to be determined. Notice that as long as $0<\kappa_i \ll 1$ is small enough for $i=1,2,3$, then $ \CE_N(V)\sim \|V\|_N^2$ holds true. Moreover, by letting $0<\kappa_3\ll \kappa_2\ll \kappa_1\ll 1$ be small enough with $\kappa_2^{3/2}\ll \kappa_3$, the sum of \eqref{thm.ap.p01}, \eqref{thm.ap.p03}$\times \kappa_1$, \eqref{thm.ap.p08}$\times \kappa_2$ and \eqref{thm.ap.p09}$\times \kappa_3$ implies that there is $\la>0$, $C>0$ such that \eqref{thm.ap.2} also holds true with $\CD_N(\cdot)$ defined in \eqref{def.ef.d}. Here, we used the following Cauchy-Schwarz inequality
\begin{equation*}
   2\kappa_2\|v\|_{N}\|\na \widetilde{B}\|_{N-2}\leq \kappa_2^{1/2}\|v\|_N^2+\kappa_2^{3/2}\|\na \widetilde{B}\|_{N-2}^2,
\end{equation*}
and due to  $\kappa_2^{3/2}\ll \kappa_3$, both terms on the r.h.s. of the above inequality were absorbed. This completes the proof of Theorem \ref{thm.ap}.
\end{proof}

\begin{remark}\label{rem.hypo}
The main idea for the proof of Theorem \ref{thm.ap}, particularly construction of the interactive functionals, is inspired by the recent studies of some degenerately dissipative kinetic equations \cite{D-Hypo,DS-VMB} and \cite{Vi}. In fact, although the nonlinear system \eqref{s.e} is degenerately dissipative, interplay between the first-order linear conservative terms and the zero-order  degenerately dissipative terms indeed yields the dissipation of all the components in the solution. This is also easier to be seen from the Fourier analysis of the linearized homogeneous system; see Theorem \ref{thm.tff} and its proof later on.

\end{remark}

\subsection{Proof of global existence} In this subsection we shall prove Proposition \ref{prop.s.exi}. Since \eqref{s.e} is a quasi-linear symmetric hyperbolic system, short-time existence follows from much more general case showed in \cite[Theorem 1.2, Proposition 1.3 and Proposition 1.4 in Chapter 16]{Ta}; see also \cite{Kato}.

\begin{lemma}[local existence]\label{lem.local}
Suppose that $V_0\in H^N(\R^3)$ satisfies \eqref{s.e.cc}. Then, there is $T_0>0$ such that the Cauchy problem \eqref{s.e}-\eqref{s.e.id} admits a unique solution on $[0,T_0)$ with
\begin{equation*}
    V\in C([0,T_0);H^{N}(\R^3))\cap Lip([0,T_0);H^{N-1}(\R^3)).
\end{equation*}
\end{lemma}

Moreover, the local solution can be extent as long as its $W^{1,\infty}$-norm is bounded; see \cite[Proposition 1.5 in Chapter 16]{Ta}.

\begin{lemma}[extension]\label{lem.extension}
Suppose that $V\in C([0,T);H^N(\R^3))$ solves the system \eqref{s.e} for $t\in (0,T)$ with $T>0$. Assume also that
\begin{equation*}
    \sup_{0\leq t<T}\|V(t)\|_{W^{1,\infty}}<\infty.
\end{equation*}
Then, there exists $T_1>T$ such that $V$ extends to a solution to  \eqref{s.e}, belonging to $C([0,T_1);H^N(\R^3))$.
\end{lemma}

\noindent{\bf Proof of Proposition \ref{prop.s.exi}:} Let $\la>0, C>0$ be defined in \eqref{thm.ap.2} and $C_2>0$ be chosen such that
\begin{equation*}
    \|\si\|_N^2\leq C_2\CE_N(V)
\end{equation*}
for $V=[\si,v,\widetilde{E},\widetilde{B}]$. Fix $\de_2>0$ such that
\begin{equation*}
    C[(2\de_2)^{1/2}+2\de_2]\leq \frac{\la}{2},\ \ 2C_2\de_2\leq 1,
\end{equation*}
and let $V_0\in H^N(\R^3)$ satisfy \eqref{s.e.cc} and $\CE_N(V_0)\leq \de_2$.
Now, let us define
\begin{equation*}
    T_\ast=\sup\left\{t\geq 0
    \left|
    \begin{array}{c}
    \exists\,V\in C([0,t);H^N(\R^3))\ \text{to the Cauchy problem}\\
    \text{\eqref{s.e}-\eqref{s.e.id} with}\ \sup\limits_{0\leq s<t}\CE_N(V(s))\leq 2\de_2
    \end{array}
    \right.\right\}.
\end{equation*}
From Lemma \ref{lem.local} and continuity of $\CE_N(V(t))$ in time, $T_\ast>0$ holds true. Suppose that $T_\ast$ is finite. Then, there exists $V\in C([0,T_\ast);H^N(\R^3))$ to the Cauchy problem \eqref{s.e}-\eqref{s.e.id} with $\sup_{0\leq s<T_\ast}\CE_N(V(s))\leq 2\de_2$. Notice that the case when $\sup_{0\leq s<T_\ast}\CE_N(V(s))< 2\de_2$ can not occur due to the definition of $T_\ast$ and Lemma \ref{lem.extension} as well as continuity of $\CE_N(V(t))$. Thus, if $T_\ast$ is finite, then
\begin{equation}\label{prop.s.exi.p1}
    \sup_{0\leq s<T_\ast}\CE_N(V(s))= 2\de_2.
\end{equation}
On the other hand, by the choices of $\de_2$ and $V_0$, it follows from Theorem \ref{thm.ap} that
\begin{equation}\label{prop.s.exi.p2}
    \sup_{0\leq t<T_\ast}\CE_N(V(t))+\frac{\la}{2}\int_0^{T_\ast} \CD_N(V(t))dt\leq \de_2.
\end{equation}
This is a contradiction to \eqref{prop.s.exi.p1}. Then, $T_\ast=\infty$ holds true. Here, we remark that although Theorem \ref{thm.ap} holds for smooth solutions, \eqref{prop.s.exi.p2} is still true for  $V\in C([0,T_\ast);H^N(\R^3))$. Finally, uniqueness of solutions and  Lipschitz continuity in  \eqref{prop.s.exi.1} follow from Lemma \ref{lem.local}, and \eqref{prop.s.exi.2} holds for any $t\geq 0$ by Theorem \ref{thm.ap} and the choice of $\de_2$.
This completes the proof of Proposition \ref{prop.s.exi}.

\section{Linearized homogeneous system}\label{sec.lhs}

In this section, in order to study in the next section the time-decay property of solutions to the nonlinear system \eqref{s.e} or \eqref{s.cr}, we are concerned with the following Cauchy problem on the linearized homogeneous system corresponding to the reformulated version \eqref{s.cr}:
\begin{equation}\label{lhs}
\left\{\begin{array}{l}
  \dis \pa_t \rho+\na \cdot u=0,\\[3mm]
  \dis \pa_t u+ \ga \na \rho+E+u=0,\\[3mm]
  \dis \pa_t E-\na \times B-u= 0,\\[3mm]
  \dis \pa_t B+\na \times E=0,\\[3mm]
  \dis \na\cdot E=-\rho,\ \ \na\cdot B=0,\ \ t>0,x\in \R^3,
\end{array}\right.
\end{equation}
with given initial data
\begin{equation}\label{lhs.id}
   U|_{t=0}=U_0:=[\rho_0,u_0,E_0,B_0],\ \ x\in \R^3,
\end{equation}
satisfying the compatible condition
\begin{equation}\label{lhs.cc}
    \na\cdot E_0=-\rho_0,\ \ \na\cdot B_0=0.
\end{equation}
Here and through this section, we always denote $U=[\rho,u,E,B]$ as the solution to the first-order hyperbolic system \eqref{lhs}. As mentioned before, we remark that in the  case of the linearized  homogeneous system, it is more convenient to consider \eqref{lhs} than the linearized version from \eqref{s.e}, and on the other hand, since smooth solutions to the nonlinear systems \eqref{s.e} and \eqref{s.cr} are equivalent, time-decay properties of the solution to \eqref{s.cr} can be directly applied to \eqref{s.e}.

The rest of this section is arranged as follows. In Section \ref{sec.lhs.1}, we derive a time-frequency Lyapunov inequality, which leads to the pointwise time-frequency upper-bound of solutions. In Section \ref{sec.lhs.2}, based on this pointwise upper-bound, we obtain the elementary $L^p$-$L^q$ time-decay property of the linear solution operator for the Cauchy problem  \eqref{lhs}-\eqref{lhs.id}. In Section \ref{sec.lhs.3}, we study the representation of the Fourier transform of solutions. In  Section \ref{sec.lhs.4}, we apply results of Section \ref{sec.lhs.3} to obtain the refined $L^p$-$L^q$ time-decay property for each component in the linear solution $[\rho,u,E,B]$ to the Cauchy problem  \eqref{lhs}-\eqref{lhs.id}.

Through this section, we also introduce some additional notations.
For an integrable function $f:
\R^3\to\R$, its Fourier transform
is defined by
\begin{equation*}
  \hat{f}(k)= \int_{\R^3} e^{-i x\cdot k} f(x)dx, \quad
  x\cdot
   k:=\sum_{j=1}^3 x_jk_j,
   \quad
   k\in\R^3,
\end{equation*}
where $i =\sqrt{-1}\in \mathbb{C}$ is the imaginary
unit. For two complex numbers or vectors $a$ and $b$, $(a\mid
b)$ denotes the dot product of $a$ with the complex conjugate of $b$.

\subsection{Time-frequency Lyapunov functional}\label{sec.lhs.1}

In this subsection, we apply the energy method in the Fourier space to the Cauchy problem \eqref{lhs}-\eqref{lhs.cc} to show that there exists a time-frequency Lyapunov functional which is equivalent with $|\hat{U}(t,k)|^2$ and moreover its dissipation rate can also be characterized by the functional itself. The method of proof is similar to that for the proof of Theorem \ref{thm.ap} in the nonlinear case. Once again, as in Remark \ref{rem.hypo}, we mention \cite{D-Hypo,DS-VMB} and \cite{Vi} for the similar idea. Let us state the main result of this subsection as follows.

\begin{theorem}\label{thm.tff}
Let $U(t,x)$, $t> 0,x\in\R^3$,  be a well-defined solution to the system \eqref{lhs}. There is a time-frequency Lyapunov functional $\CE(\hat{U}(t,k))$ with
\begin{equation}\label{thm.tff.1}
  \CE(\hat{U})\sim |\hat{U}|^2:=|\hat{\rho}|^2+|\hat{u}|^2+|\hat{E}|^2+|\hat{B}|^2
\end{equation}
satisfying that there is $\la>0$ such that the Lyapunov inequality
\begin{equation}\label{thm.tff.2}
    \frac{d}{dt}\CE(\hat{U}(t,k))+\frac{\la |k|^2}{(1+|k|^2)^2}\CE(\hat{U}(t,k))\leq 0
\end{equation}
holds for any $t> 0$ and $k\in \R^3$.
\end{theorem}

\begin{proof}
It is based on the Fourier analysis of the system \eqref{lhs}. For that, after taking Fourier transform in $x$ for \eqref{lhs}, $\hat{U}=[\hat{\rho},\hat{u},\hat{E},\hat{B}]$ satisfies
\begin{equation}\label{lhs.f}
\left\{\begin{array}{l}
  \dis \pa_t \hat{\rho}+i k\cdot \hat{u}=0,\\[1mm]
  \dis \pa_t \hat{u}+ \ga i k \hat{\rho}+\hat{E}+\hat{u}=0,\\[1mm]
  \dis \pa_t \hat{E}-ik \times \hat{B}-\hat{u}= 0,\\[1mm]
  \dis \pa_t \hat{B}+i k \times \hat{E}=0,\\[1mm]
  \dis i k\cdot \hat{E}=-\hat{\rho},\ \ k\cdot \hat{B}=0,\ \ t>0,k\in \R^3.
\end{array}\right.
\end{equation}
First of all, it is straightforward to obtain from the first four equations of \eqref{lhs.f} that
\begin{equation}\label{thm.tff.p1}
    \frac{1}{2}\pa_t |[\sqrt{\ga} \hat{\rho},\hat{u},\hat{E},\hat{B}]|^2+|\hat{u}|^2=0.
\end{equation}
By taking the complex dot product of the second equation of \eqref{lhs.f} with $ik\hat{\rho}$, using integration by parts in $t$ and then replacing $\pa_t\hat{\rho}$ by the first equation of \eqref{lhs.f}, one has
\begin{equation*}
    \pa_t (\hat{u}\mid ik \hat{\rho})+(1+\ga |k|^2)|\hat{\rho}|^2 =|k\cdot \hat{u}|^2-(\hat{u}\mid i k\hat{\rho}),
\end{equation*}
which by taking the real part and using the Cauchy-Schwarz inequality, implies
\begin{equation*}
    \pa_t \FR (\hat{u}\mid ik \hat{\rho})+\la (1+|k|^2)|\hat{\rho}|^2 \leq C(1+|k|^2)|\hat{u}|^2.
\end{equation*}
Dividing it by $1+|k|^2$ gives
\begin{equation}\label{thm.tff.p2}
    \pa_t \frac{\FR (\hat{u}\mid ik \hat{\rho})}{1+|k|^2}+\la |\hat{\rho}|^2 \leq C|\hat{u}|^2.
\end{equation}
In a similar way, by taking the complex dot product of the second equation of \eqref{lhs.f} with $\hat{E}$, using integration by part in $t$ and then replacing $\pa_t\hat{E}$ by the third equation of \eqref{lhs.f}, one has
\begin{equation}\label{thm.tff.p3}
    \pa_t (\hat{u}\mid \hat{E})+\ga |k\cdot \hat{E}|^2+|\hat{E}|^2=-(\hat{u}\mid \hat{E})+(\hat{u}\mid ik\times \hat{B})+|\hat{u}|^2,
\end{equation}
where we used $ik\cdot \hat{E}=-\hat{\rho}$ to obtain
\begin{equation*}
    (\ga ik\hat{\rho}\mid \hat{E})=\ga(-ik\cdot \hat{E}\mid ik\cdot \hat{E})=\ga |k\cdot \hat{E}|^2.
\end{equation*}
Taking the real part of \eqref{thm.tff.p3} and using the Cauchy-Schwarz inequality implies
\begin{equation*}
    \pa_t \FR (\hat{u}\mid \hat{E})+\la (|k\cdot \hat{E}|^2+|\hat{E}|^2)\leq C|\hat{u}|^2+\FR (\hat{u}\mid ik\times \hat{B}),
\end{equation*}
which further multiplying it by $|k|^2/(1+|k|^2)^2$ gives
\begin{equation}\label{thm.tff.p4}
    \pa_t \frac{|k|^2\FR (\hat{u}\mid \hat{E})}{(1+|k|^2)^2}+\frac{\la |k|^2(|k\cdot \hat{E}|^2+|\hat{E}|^2)}{(1+|k|^2)^2}\leq C|\hat{u}|^2+\frac{ |k|^2\FR (\hat{u}\mid ik\times \hat{B})}{(1+|k|^2)^2}.
\end{equation}
Similarly, it follows from equations of the  electromagnetic field in \eqref{lhs.f} that
\begin{equation*}
    \pa_t (-ik\times \hat{B}\mid \hat{E})+|k\times \hat{B}|^2=|k\times \hat{E}|^2-(ik\times \hat{B}\mid \hat{u}),
\end{equation*}
which after using Cauchy-Schwarz and dividing it by $(1+|k|^2)^2$, implies
\begin{equation}\label{thm.tff.p5}
     \pa_t \frac{\FR (-ik\times \hat{B}\mid \hat{E})}{(1+|k|^2)^2}+\frac{\la |k\times \hat{B}|^2}{(1+|k|^2)^2}\leq
     \frac{ |k|^2|\hat{E}|^2}{(1+|k|^2)^2}+C|\hat{u}|^2.
\end{equation}
Finally, let us define
\begin{eqnarray*}
  \CE(\hat{U}(t,k)) &=& |[\sqrt{\ga} \hat{\rho},\hat{u},\hat{E},\hat{B}]|^2+ \kappa_1  \frac{\FR (\hat{u}\mid ik \hat{\rho})}{1+|k|^2}+\kappa_2 \frac{\FR (|k|^2\hat{u}\mid \hat{E})}{(1+|k|^2)^2}\\
  &&+\kappa_3 \frac{\FR (-ik\times \hat{B}\mid \hat{E})}{(1+|k|^2)^2}
\end{eqnarray*}
for constants $0< \kappa_3\ll \kappa_2\ll \kappa_1\ll 1$ to be chosen. Let $0<\kappa_i\ll 1$, $i=1,2,3$, be small enough such that \eqref{thm.tff.1} holds true. On the other hand, by letting $0< \kappa_3\ll \kappa_2\ll \kappa_1\ll 1$ be further small enough with $\kappa_2^{3/2}\ll \kappa_3$, the sum of \eqref{thm.tff.p1}, \eqref{thm.tff.p2}$\times \kappa_1$, \eqref{thm.tff.p4}$\times \kappa_2$ and \eqref{thm.tff.p5}$\times \kappa_3$ gives
\begin{equation}\label{thm.tff.p6}
    \pa_t \CE(\hat{U}(t,k))+\la |[\hat{\rho},\hat{u}]|^2+\frac{\la |k|^2}{(1+|k|^2)^2}|[\hat{E},\hat{B}]|^2\leq 0,
\end{equation}
where we used the identity $|k\times \hat{B}|^2=|k|^2|\hat{B}|^2$ due to $k\cdot \hat{B}=0$ and also used the following Cauchy-Schwarz inequality
\begin{equation*}
    \frac{ \kappa_2|k|^2\FR (\hat{u}\mid ik\times \hat{B})}{(1+|k|^2)^2}\leq  \frac{ \kappa_2^{1/2}|k|^4|\hat{u}|^2}{2(1+|k|^2)^2}+\frac{ \kappa_2^{3/2}|k|^2|\hat{B}|^2}{2(1+|k|^2)^2}.
\end{equation*}
Therefore, \eqref{thm.tff.2} follows from \eqref{thm.tff.p6} by noticing $\CE(\hat{U}(t,k))\sim |\hat{U}|^2$ and
\begin{equation*}
|[\hat{\rho},\hat{u}]|^2+\frac{|k|^2}{(1+|k|^2)^2}|[\hat{E},\hat{B}]|^2\geq \frac{\la |k|^2}{(1+|k|^2)^2}|\hat{U}|^2.
\end{equation*}
This completes the proof of Theorem \ref{thm.tff}.
\end{proof}

Theorem \ref{thm.tff} directly leads to the pointwise time-frequency estimate on the modular $|\hat{U}(t,k)|$ in terms of initial data modular $|\hat{U}_0(k)|$.

\begin{corollary}
Let $U(t,x)$, $t\geq 0,x\in\R^3$,  be a well-defined solution to the Cauchy problem \eqref{lhs}-\eqref{lhs.cc}. Then, there are $\la>0$, $C>0$ such that
\begin{equation}\label{cor.ptf.est}
   |\hat{U}(t,k)|\leq C e^{-\frac{\la |k|^2 t}{(1+|k|^2)^2}}|\hat{U}_0(k)|
\end{equation}
holds for any $t\geq  0$ and $k\in \R^3$.
\end{corollary}

\subsection{$L^p$-$L^q$ time-decay property}\label{sec.lhs.2}

In this subsection we study the $L^p$-$L^q$ time-decay property of the solution $U$ to the Cauchy problem \eqref{lhs}-\eqref{lhs.id} on the basis of the pointwise time-frequency estimate \eqref{cor.ptf.est}. The refined  $L^p$-$L^q$ estimates on each component in $U$ will be given Section \ref{sec.lhs.4}.
Formally, the solution to the Cauchy problem \eqref{lhs}-\eqref{lhs.id} is denoted by
\begin{equation}
U(t)=e^{tL}U_0,\label{def.lu}
\end{equation}
where $e^{tL}$, $t\geq 0$, is called the linear solution operator. The main result of this subsection is stated as follows.

\begin{theorem}\label{thm.lhs}
Let $1\leq p,r\leq 2\leq q\leq \infty$, $\ell\geq 0$ and let $m\geq 0$ be an integer. Define
\begin{equation}\label{def.index}
 [\ell+3(\frac{1}{r}-\frac{1}{q})]_+\\
 =\left\{
\begin{array}{ll}
  [\ell+3(\frac{1}{r}-\frac{1}{q})]_-+1 & \ \ \text{when $r\neq 2$ or $q\neq 2$}\\
  & \ \ \text{or $\ell$ is not an integer},\\[3mm]
  \ell & \ \ \text{when $r=q=2$}\\
  & \ \ \text{and $\ell$ is an integer},\\
\end{array}\right.
\end{equation}
where $[\cdot]_-$ denotes the integer part of the argument. Suppose $U_0$ satisfies \eqref{lhs.cc}. Then, $e^{tL}$ satisfies the following time-decay property:
\begin{multline}
\label{thm.lhs.1}
\|\na^m e^{tL}U_0\|_{L^q}
\leq C(1+t)^{-\frac{3}{2}(\frac{1}{p}-\frac{1}{q})-\frac{m}{2}}\|U_0\|_{L^p}\\
+C(1+t)^{-\frac{\ell}{2}}
\|\na^{m+[\ell+3(\frac{1}{r}-\frac{1}{q})]_+} U_0\|_{L^r}
\end{multline}
for any $t\geq 0$, where $C=C(p,q,r,\ell,m)$.
\end{theorem}

\begin{proof}
Take $2\leq q\leq \infty$ and an integer $m\geq 0$. Set $U(t)=e^{t L}U_0$. From Hausdorff-Young inequality,
\begin{eqnarray}\label{thm.lhs.p1}
    \|\na^m U(t) \|_{L^q(\R^3_x)}&\leq& C\left\||k|^m\hat{U}(t)\right\|_{L^{q'}(\R^3_k)}\\
    &\leq& C \left\||k|^m\hat{U}(t)\right\|_{L^{q'}(|k|\leq 1)}+C\left\||k|^m\hat{U}(t)\right\|_{L^{q'}(|k|\geq 1)},\nonumber
\end{eqnarray}
where $\frac{1}{q}+\frac{1}{q'}=1$. Notice that using the lower bounds
\begin{equation*}
   \frac{|k|^2 }{(1+|k|^2)^2}\geq  \frac{|k|^2}{2} \ \ \text{if}\ |k|\leq 1, \ \text{and}\ \  \frac{|k|^2 }{(1+|k|^2)^2}\geq  \frac{1}{4|k|^2} \ \ \text{if}\ |k|\geq 1,
\end{equation*}
it follows from \eqref{cor.ptf.est} that
\begin{equation*}
    |\hat{U}(t,k)|\leq\left\{
    \begin{array}{ll}
      Ce^{- \frac{\la}{2} |k|^2 t} |\hat{U_0}(k)| & \ \ \text{if}\ |k|\leq 1,\\[3mm]
      Ce^{- \frac{\la}{4 |k|^2 t}} |\hat{U_0}(k)| & \ \ \text{if}\ |k|\geq 1.
    \end{array}\right.
\end{equation*}
Thus, as in \cite{Ka} or \cite{HZ},
\begin{equation}\label{thm.lhs.p2}
   \left\||k|^m\hat{U}(t)\right\|_{L^{q'}(|k|\leq 1)}\leq   C(1+t)^{-\frac{3}{2}(\frac{1}{p}-\frac{1}{q})-\frac{m}{2}}\|U_0\|_{L^p}
\end{equation}
for any $1\leq p\leq 2$. On the other hand, letting $\ell\geq 0$, one has
\begin{equation*}
 \left\||k|^m\hat{U}(t)\right\|_{L^{q'}(|k|\geq 1)}\leq \sup_{|k|\geq 1} \left(\frac{1}{|k|^\ell}e^{-\frac{\la t}{4|k|^2}}\right) \left\||k|^{m+\ell}\hat{U}_0\right\|_{L^{q'}(|k|\geq 1)}
\end{equation*}
Since
\begin{equation*}
    \sup_{|k|\geq 1} \left(\frac{1}{|k|^\ell}e^{-\frac{\la t}{4|k|^2}}\right) \leq C(1+t)^{-\frac{\ell}{2}},
\end{equation*}
it follows that
\begin{equation}\label{thm.lhs.p3}
 \left\||k|^m\hat{U}(t)\right\|_{L^{q'}(|k|\geq 1)}\leq C(1+t)^{-\frac{\ell}{2}} \left\||k|^{m+\ell}\hat{U}_0\right\|_{L^{q'}(|k|\geq 1)}.
\end{equation}
Now, take $1\leq r\leq 2$ and fix $\eps>0$ small enough. By H\"{o}lder inequality $1/q'=1/r'+(r'-q')/(r'q')$ with $\frac{1}{r}+\frac{1}{r'}=1$,
\begin{multline}\label{thm.lhs.p4}
 \left\||k|^{m+\ell}\hat{U}_0\right\|_{L^{q'}(|k|\geq 1)}= \left\||k|^{-\frac{r'-q'}{r'q'}(3+\eps)}|k|^{m+\ell+\frac{r'-q'}{r'q'}(3+\eps)}\hat{U}_0\right\|_{L^{q'}(|k|\geq 1)}\\
 \leq \left\||k|^{-(3+\eps)}\right\|_{L^1(|k|\geq 1)}^{\frac{r'-q'}{r'q'}}
 \left\||k|^{m+\ell+\frac{r'-q'}{r'q'}(3+\eps)}\hat{U}_0\right\|_{L^{r'}(|k|\geq 1)}\\
 \leq C
 \left\||k|^{m+\ell+(\frac{1}{r}-\frac{1}{q})(3+\eps)}\hat{U}_0\right\|_{L^{r'}(|k|\geq 1)}.
\end{multline}
When $r=q=2$ and $\ell$ is an integer,
\begin{equation*}
 \left\||k|^{m+\ell+(\frac{1}{r}-\frac{1}{q})(3+\eps)}\hat{U}_0\right\|_{L^{r'}(|k|\geq 1)}=   \left\||k|^{m+\ell}\hat{U}_0\right\|_{L^{2}(|k|\geq 1)}\leq C\|\na^{m+\ell}U_0\|,
\end{equation*}
which after plugging into \eqref{thm.lhs.p4} and then \eqref{thm.lhs.p3}, together with \eqref{thm.lhs.p2} and \eqref{thm.lhs.p1}, implies \eqref{thm.lhs.1}. When $r\neq 2$ or $q\neq 2$ or $\ell$ is not an integer, by letting $\eps>0$ small enough, it follows from Hausdorff-Young inequality that
\begin{multline*}
\left\||k|^{m+\ell+(\frac{1}{r}-\frac{1}{q})(3+\eps)}\hat{U}_0\right\|_{L^{r'}(|k|\geq 1)}
\leq \left\||k|^{m+[\ell+3(\frac{1}{r}-\frac{1}{q})]_-+1}\hat{U}_0\right\|_{L^{r'}(|k|\geq 1)}\\
\leq C\|\na^{m+[\ell+3(\frac{1}{r}-\frac{1}{q})]_+} U_0\|_{L^r},
\end{multline*}
which, similarly after plugging into \eqref{thm.lhs.p4} and then \eqref{thm.lhs.p3}, together with \eqref{thm.lhs.p2}, implies \eqref{thm.lhs.1}. This completes the proof of Theorem \ref{thm.lhs}.
\end{proof}


\subsection{Representation of solutions}\label{sec.lhs.3}

In this subsection, we furthermore explore the explicit solution $U=[\rho,u,E,B]=e^{tL}U_0$ to the Cauchy problem \eqref{lhs}-\eqref{lhs.id} with the condition \eqref{s.cr.cc} or equivalently the system \eqref{lhs.f} in the time-frequency variables. The main goal is to prove Theorem \ref{thm.green} stated at the end of this subsection.

Taking the time derivative for the first equation of \eqref{lhs} and using the second equation of \eqref{lhs} to replace $\pa_tu$, it follows that
\begin{equation*}
    \pa_{tt}\rho-\ga \De\rho-\na\cdot E-\na\cdot u=0.
\end{equation*}
Further noticing $\na\cdot E=-\rho$ and $\na\cdot u=-\pa_t\rho$, one has
\begin{equation}\label{a.rho.p1}
    \pa_{tt}\rho-\ga\De\rho+\rho+\pa_t\rho=0.
\end{equation}
Initial data is given by
\begin{equation}\label{a.rho.p2}
 \rho|_{t=0}=\rho_0=-\na\cdot E_0,\ \ \pa_t\rho|_{t=0}=-\na\cdot u_0.
\end{equation}
By solving the Fourier transform of the second order ODE \eqref{a.rho.p1}-\eqref{a.rho.p2} as
\begin{equation*}
 \left\{\begin{array}{l}
    \dis \pa_{tt}\hat{\rho}+(1+\ga |k|^2)\hat{\rho}+\pa_t\hat{\rho}=0,\\[2mm]
     \dis \hat{\rho}|_{t=0}=\hat{\rho}_0=-ik\cdot \hat{E}_0,\\[2mm]
     \dis \pa_t\hat{\rho}|_{t=0}=-ik\cdot \hat{u}_0,
 \end{array}\right.
\end{equation*}
it is easy to obtain
\begin{eqnarray}
  \hat{\rho}(t,k) &=& \hat{\rho}_0e^{-\frac{t}{2}}\cos(\sqrt{3/4+\ga|k|^2} t) \label{a.rho.p3}\\
  &&+(\frac{1}{2}\hat{\rho}_0-ik\hat{u}_0)e^{-\frac{t}{2}}\frac{\sin(\sqrt{3/4+\ga|k|^2} t)}{\sqrt{3/4+\ga|k|^2}}.\nonumber
\end{eqnarray}
Again using $\na\cdot E=-\rho$, \eqref{a.rho.p3} implies
\begin{eqnarray*}
  \tilde{k}\cdot\hat{E}(t,k) &=& \tilde{k}\cdot \hat{E}_0e^{-\frac{t}{2}}\cos(\sqrt{3/4+\ga|k|^2} t) \\
  &&+\tilde{k}\cdot (\frac{1}{2}\hat{E}_0+\hat{u}_0)e^{-\frac{t}{2}}\frac{\sin(\sqrt{3/4+\ga|k|^2} t)}{\sqrt{3/4+\ga|k|^2}}.\nonumber
\end{eqnarray*}
Here and in the sequel we set $\tilde{k}=k/|k|$ for $|k|\neq 0$. Similarly, taking the time derivative for the second equation of \eqref{lhs} and then replacing $\pa_t \rho$, $\pa_t E$ by the first and third equations of \eqref{lhs}, it follows that
\begin{equation*}
    \pa_{tt}u-\ga \na \na\cdot u+\na\times B+u+\pa_tu=0.
\end{equation*}
Further taking the divergence, one has
\begin{equation}\label{a.u.div.p1}
    \pa_{tt}(\na\cdot u)-\ga \De \na\cdot u+\na\cdot u+\pa_t\na\cdot u=0.
\end{equation}
Notice
\begin{eqnarray}
 \na\cdot u|_{t=0}&=&\na\cdot u_0,\label{a.u.div.p2}\\
\pa_t\na\cdot u|_{t=0}&=&-\ga\De\rho_0-\na\cdot E_0-\na\cdot u_0=-\ga\De\rho_0+\rho_0-\na\cdot u_0.\label{a.u.div.p3}
\end{eqnarray}
Similarly, by solving the Fourier transform of the second ODE \eqref{a.u.div.p1} with  \eqref{a.u.div.p2}-\eqref{a.u.div.p3} as
\begin{equation*}
 \left\{\begin{array}{l}
    \dis \pa_{tt}(\tilde{k}\cdot \hat{u})+(1+\ga |k|^2)(\tilde{k}\cdot \hat{u})+\pa_t(\tilde{k}\cdot \hat{u})=0,\\[2mm]
     \dis (\tilde{k}\cdot \hat{u})|_{t=0}=\tilde{k}\cdot \hat{u}_0,\\[2mm]
     \dis \pa_t(\tilde{k}\cdot \hat{u})|_{t=0}=\tilde{k}\cdot (-i\ga k\hat{\rho}_0-\hat{E}_0-\hat{u}_0),
 \end{array}\right.
\end{equation*}
one has
\begin{eqnarray*}
  \tilde{k}\cdot \hat{u}(t,k) &=& \tilde{k}\cdot \hat{u}_0e^{-\frac{t}{2}}\cos(\sqrt{3/4+\ga|k|^2} t) \\
  &&+\tilde{k}\cdot (-\frac{1}{2}\hat{u}_0-i\ga k \hat{\rho}_0-\hat{E}_0)e^{-\frac{t}{2}}\frac{\sin(\sqrt{3/4+\ga|k|^2} t)}{\sqrt{3/4+\ga|k|^2}}.
\end{eqnarray*}

Next, we shall solve
\begin{equation*}
\left\{\begin{array}{l}
   M_1(t,k):= -\tilde{k}\times (\tilde{k}\times \hat{u}(t,k)),\\
    M_2(t,k):=-\tilde{k}\times (\tilde{k}\times \hat{E}(t,k)),\\
    M_3(t,k):=-\tilde{k}\times (\tilde{k}\times \hat{B}(t,k)),
\end{array}\right.
\end{equation*}
for $t>0$ and $|k|\neq 0$. Taking the curl for the equations of $\pa_tu,\pa_t E,\pa_tB$ in \eqref{lhs}, it follows that
\begin{equation*}
    \left\{\begin{array}{l}
      \pa_t (\na\times u)+\na\times E+\na\times u=0,\\
      \pa_t (\na\times E)-\na\times (\na\times B)-\na\times u=0,\\
      \pa_t (\na\times B)+\na\times (\na\times E)=0.
    \end{array}\right.
\end{equation*}
In terms of the Fourier transform in $x$,  one has
\begin{equation}\label{a.SM}
    \left\{\begin{array}{rrrr}
      \pa_tM_1 =& -M_1 & -M_2 &\\
      \pa_tM_2 =&M_1 &&+ik\times M_3\\
      \pa_tM_3 =& &-ik\times M_2& \\
    \end{array}\right.
\end{equation}
with initial data
\begin{equation}\label{a.SM.id}
    [M_1,M_2,M_3]|_{t=0}=[M_{1,0},M_{2,0},M_{3,0}].
\end{equation}
Here, we have defined
\begin{equation*}
    M_{1,0}= -\tilde{k}\times(\tilde{k}\times \hat{u}_0),\ \
    M_{2,0}= -\tilde{k}\times(\tilde{k}\times \hat{E}_0),\ \
    M_{3,0}= -\tilde{k}\times(\tilde{k}\times \hat{B}_0).
\end{equation*}
Taking the time derivative for the second equation of \eqref{a.SM} and then using the other two equations to replace $\pa_tM_1$ and $\pa_tM_3$ gives
\begin{equation*}
    \pa_{tt}M_2=-M_1-M_2+k\times (k\times M_2),
\end{equation*}
which from $k\times (k\times M_2)=-|k|^2M_2$ due to $k\cdot M_2=0$, implies
\begin{equation}\label{a.SM.p02}
    \pa_{tt}M_2+(1+|k|^2)M_2=-M_1.
\end{equation}
Further taking the time derivative for \eqref{a.SM.p02} and replacing $\pa_tM_1$ by the first equation of \eqref{a.SM}, one has
\begin{equation}\label{a.SM.p03}
    \pa_{ttt}M_2+(1+|k|^2)\pa_tM_2=-\pa_tM_1=M_1+M_2.
\end{equation}
The sum of \eqref{a.SM.p02} and \eqref{a.SM.p03} yields the following three order ODE for $M_2$:
\begin{equation}\label{a.SM.p04}
    \pa_{ttt}M_2+\pa_{tt}M_2+(1+|k|^2)\pa_tM_2+|k|^2M_2=0.
\end{equation}
Initial data is given as
\begin{equation}\label{a.SM.p04.id}
    \left\{\begin{array}{rl}
      M_2|_{t=0}=&M_{2,0},\\[2mm]
      \pa_tM_2|_{t=0}=&M_{1,0}+ik\times M_{3,0},\\[2mm]
      \pa_{tt}M_2|_{t=0}=&-M_{1,0}-(1+|k|^2)M_{2,0}.
    \end{array}\right.
\end{equation}
The characteristic equation of \eqref{a.SM.p04} reads
\begin{equation*}
    F(\chi):=\chi^3+\chi^2+(1+|k|^2)\chi+|k|^2=0.
\end{equation*}
For the roots of the above characteristic equation and their basic properties, one has

\begin{lemma}\label{a.lem.root}
Let $|k|\neq 0$. The equation $F(\chi)=0$, $\chi\in \C$, has a real root $\si=\si(|k|)\in (-1,0)$ and two conjugate complex roots $\chi_{\pm}=\be\pm i\om$ with $\be=\be(|k|)\in (-1/2,0)$ and $\om=\om(|k|)\in(\sqrt{6}/3,\infty)$ satisfying
\begin{equation}\label{a.lem.root.1}
    \be=-\frac{\si+1}{2},\ \ \om=\frac{1}{2} \sqrt{3\si^2+2\si+3+4|k|^2}.
\end{equation}
$\si,\be,\om$ are smooth over $|k|>0$, and $\si(|k|)$ is strictly decreasing in $|k|>0$ with
\begin{equation*}
    \lim\limits_{|k|\to 0}\si(|k|)=0,\ \ \ \lim\limits_{|k|\to \infty}\si(|k|)=-1.
\end{equation*}
Mover, the following asymptotic behaviors hold true:
\begin{eqnarray*}
&\dis \si(|k|)=-O(1)|k|^2,\ \ \be(|k|)=-\frac{1}{2}+O(1)|k|^2,\ \ \om(|k|)=\frac{\sqrt{3}}{2}+O(1)|k|
\end{eqnarray*}
whenever $|k|\leq 1$ is small, and
\begin{eqnarray*}
&\dis \si(|k|)=-1+O(1)|k|^{-2},\ \ \be(|k|)=-O(1)|k|^{-2},\ \ \om(|k|)=O(1)|k|
\end{eqnarray*}
whenever $|k|\geq 1$ is large. Here and in the sequel $O(1)$ denotes a generic strictly positive constant.
\end{lemma}

\begin{proof}
Suppose $|k|\neq 0$. Let us first find the possibly existing real root for equation $F(\chi)=0$ over $\chi\in\R$. Notice that
\begin{equation*}
    F'(\chi)=3\chi^2+2\chi +(1+|k|^2)=3(\chi+\frac{1}{3})^2+(\frac{2}{3}+|k|^2)>0
\end{equation*}
and $F(0)=|k|^2>0$, $F(-1)=-1<0$, then equation $F(\chi)=0$ indeed has one and only one real root denoted by $\si=\si(|k|)$ satisfying $-1<\si<0$. Since $F(\cdot)$ is smooth, then $\si(\cdot)$ is also smooth in $|k|>0$. By taking derivative of $F(\si(|k|))=0$ in $|k|$, one has
\begin{equation*}
    \si'(|k|)=\frac{-2|k|[1+\si(k)]}{ 3 [\si(k)]^2+2 \si(|k|)+|k|^2+1}<0,
\end{equation*}
so that $\si(\cdot)$ is strictly decreasing in $|k|>0$. Since $F(\si)=0$ can be re-written as
\begin{equation*}
    \si\left[\frac{\si (1+\si)}{1+|k|^2}+1\right]=-\frac{|k|^2}{1+|k|^2},
\end{equation*}
then $\si$ has limits $0$ and $-1$ as $|k|\to 0$ and $|k|\to \infty$, respectively, and moreover
$\si=-O(1)|k|^2$ whenever $|k|\leq 1$ is small.  $F(\si)=0$ is also equivalent with
\begin{equation*}
    \si+1=\frac{1}{\si^2+1+|k|^2}.
\end{equation*}
Therefore, it follows that $\si=-1+O(1)|k|^{-2}$ whenever $|k|\geq 1$ is large.

Next, let us find roots of $F(\chi)=0$ over $\chi\in \C$. Since $F(\si)=0$ with $\si\in \R$, $F(\chi)=0$ can be factored as
\begin{equation*}
    F(\chi)=(\chi-\si)[(\chi-\si)^2+(3\si+1)(\chi-\si)+3\si^2+2\si+|k|^2+1]=0.
\end{equation*}
Then, two conjugate complex roots $\chi_{\pm}=\be\pm i\om$ turn out to exist and satisfy
\begin{equation*}
  (\chi-\si)^2+(3\si+1)(\chi-\si)+3\si^2+2\si+|k|^2+1=0.
\end{equation*}
It follows that $\be=\be(|k|),\om=\om(|k|)$ take the form of \eqref{a.lem.root.1} by solving the above equation. Notice that the asymptotic behavior of $\om(|k|)$, $\be(|k|)$ at $|k|=0$ and $\infty$ directly results from that of $\si(|k|)$. This completes the proof of Lemma \ref{a.lem.root}.
\end{proof}

From Lemma \ref{a.lem.root}, one can set the solution of \eqref{a.SM.p04} as
\begin{equation}\label{def.M2}
    M_2(t,k)=c_1(k)e^{\si t}+e^{\be t}[c_2(k)\cos \om t+c_3(k)\sin \om t],
\end{equation}
where $c_i(k),1\leq i\leq 3$, are to be determined by \eqref{a.SM.p04.id} later. In fact, \eqref{def.M2} implies
\begin{equation}\label{a.SM.p06}
    \left[\begin{array}{r}
       M_2|_{t=0}\\ \pa_t M_2|_{t=0} \\ \pa_{tt}M_2|_{t=0}
    \end{array}\right]
    =A \left[\begin{array}{c}
       c_1\\ c_2 \\ c_3
    \end{array}\right],\ \ A:=
     \left[\begin{array}{rrr}
       \FI_3  &\ \  \FI_3   &\ \ \FO_3\\
       \si \FI_3 &\ \ \be \FI_3 &\ \ \om \FI_3\\
       \si^2 \FI_3 &\ \ (\be^2-\om^2)\FI_3 &\ \ 2\be\om\FI_3
    \end{array}\right].
\end{equation}
It is straightforward to check that
\begin{equation*}
    {\rm det}A=\om [\om^2+(\si-\be)^2]=\om(3\si^2+2\si+1+|k|^2)>0
\end{equation*}
and
\begin{equation*}
A^{-1}=\frac{1}{{\rm det}A}
     \left[\begin{array}{rrr}
      (\be^2+\om^2)\om \FI_3  &\ \   -2\be\om\FI_3   &\ \ \om\FI_3\\
        \si(\si-2\be)\om\FI_3 &\ \  2\be\om\FI_3 &\ \ -\om \FI_3\\
       \si(\be^2+\om^2-\si\be) \FI_3 &\ \  (\om^2+\si^2-\be^2)\FI_3 &\ \ (\be-\si)\FI_3
    \end{array}\right].
\end{equation*}
Notice that \eqref{a.SM.p06} together with \eqref{a.SM.p04.id} gives
\begin{equation*}
 \left[\begin{array}{c}
       c_1\\ c_2 \\ c_3
    \end{array}\right]=A^{-1}
     \left[\begin{array}{rrr}
       \FO_3  &\ \  \FI_3   &\ \ \FO_3\\
        \FI_3 &\ \ \FO_3 &\ \ ik\times\\
      -\FI_3 &\ \ -(1+|k|^2)\FI_3 &\ \ \FO_3
    \end{array}\right]
    \left[\begin{array}{c}
       M_{1,0}\\ M_{2,0} \\ M_{3,0}
    \end{array}\right],
\end{equation*}
which after plugging $A^{-1}$, implies
\begin{align*}
& [c_1,c_2,c_3]^{T}= \frac{1}{{\rm det}A}\\[2mm]
   & \left[\begin{array}{rrr}
       -(2\be+1)\om\FI_3  &   (\be^2+\om^2-|k|^2-1)\om\FI_3   &\ -2\be \om ik\times\\[3mm]
        (2\be+1)\om\FI_3 &  (\si^2-2\si\be+|k|^2+1)\om\FI_3 &  2\be \om ik\times\\[3mm]
        \begin{array}{r}
          (\si^2+\om^2-\be^2\ \,\\
           -\be+\si)\FI_3
        \end{array}
       &  \begin{array}{r}
       [\si(\be^2-\om^2-\si\be)\ \ \ \ \ \\
        -(\be-\si)(1+|k|^2)]\FI_3
        \end{array}
      &  \begin{array}{r}
             (\si^2+\om^2\ \ \ \ \\
              -\be^2)ik\times
         \end{array}
    \end{array}\right]
    \left[\begin{array}{c}
       M_{1,0}\\[4mm] M_{2,0} \\[8mm] M_{3,0}
    \end{array}\right].
\end{align*}
Here $[\cdot]^{T}$ denotes the transpose of a vector. Using the form of $\be$ and $\om$ to make further simplifications, one has
\begin{align}\label{def.C123}
&  [c_1,c_2,c_3]^{T}=\frac{1}{3\si^2+2\si+1+|k|^2}\\
    &\,
     \left[\begin{array}{rrr}
       \si \FI_3 &   \si(\si+1)\FI_3   & (\si+1)ik\times\\
        - \si\FI_3 &  (2\si^2+\si+|k|^2+1)\FI_3 & -(\si+1)ik\times\\
      \frac{\frac{3}{2}\si^2+\frac{3}{2}\si+1+|k|^2}{\om}\FI_3 & \frac{(\si+1)(\si+1+|k|^2)}{2\om}\FI_3 &\ \ \frac{\frac{3}{2}\si^2+\frac{1}{2}+|k|^2}{\om}ik\times
    \end{array}\right]
    \left[\begin{array}{c}
       M_{1,0}\\ M_{2,0} \\ M_{3,0}
    \end{array}\right].\nonumber
\end{align}
Now, in order to get $M_1(t,k)$ and $M_3(t,k)$ from $M_2(t,k)$, it follows from the first and third equations of \eqref{a.SM} that
\begin{eqnarray*}
  M_1(t,k) &=& M_{1,0}(k)e^{-t}-\int_0^{t}e^{-(t-s)}M_2(s,k)ds,\\
  M_3(t,k) &=& M_{3,0}(k)-ik\times \int_0^tM_2(s,k)ds.
\end{eqnarray*}
Putting \eqref{def.M2} into the above equations and taking integrations in time gives
\begin{eqnarray*}
  M_1(t,k) &=& [M_{1,0}(k)+c_4(k)]e^{-t}-\frac{c_1(k)}{1+\si}e^{\si t}\\
  &&-\frac{c_2(k)}{(1+\be)^2+\om^2}e^{\be t}\left[(1+\be)\cos \om t +\om\sin \om  t\right]\\
    &&-\frac{c_3(k)}{(1+\be)^2+\om^2}e^{\be t}\left[(1+\be)\sin \om t -\om\cos \om  t\right]
\end{eqnarray*}
and
\begin{eqnarray*}
  M_3(t,k) &=& [M_{3,0}(k)+ik\times c_5(k)]-ik\times \frac{c_1(k)}{\si}e^{\si t}\\
  &&-ik\times\frac{c_2(k)}{\be^2+\om^2}e^{\be t}\left[\be\cos \om t +\om\sin \om  t\right]\\
    &&-ik\times \frac{c_3(k)}{\be^2+\om^2}e^{\be t}\left[\be\sin \om t -\om\cos \om  t\right]
\end{eqnarray*}
where $c_4(k)$, $c_5(k)$ are chosen such that $[M_1,M_3]|_{t=0}=[M_{1,0},M_{3,0}]$ by \eqref{a.SM.id}  and hence
\begin{eqnarray*}
  c_4(k) &=& \frac{1}{(1+\si)[(1+\be)^2+\om^2]}\\
  &&[(1+\be)^2+\om^2, (1+\be)(1+\si),-\om(1+\si)][c_1,c_2,c_3]^{T}, \\
  c_5(k) &=&\frac{1}{\si(\be^2+\om^2)}[\be^2+\om^2, \si\be,-\si\om][c_1,c_2,c_3]^{T}.
\end{eqnarray*}
Notice that after tenuous computations, one can check that
\begin{eqnarray*}
&&M_{1,0}(k)+c_4(k)=0,\\
 && M_{3,0}(k)+ik\times c_5(k)= 0,
\end{eqnarray*}
for all $|k|\neq 0$. Then,
\begin{eqnarray}\label{def.M1}
  M_1(t,k) &=& -\frac{c_1(k)}{1+\si}e^{\si t}\\
  &&-\frac{c_2(k)}{(1+\be)^2+\om^2}e^{\be t}\left[(1+\be)\cos \om t +\om\sin \om  t\right]\nonumber\\
    &&-\frac{c_3(k)}{(1+\be)^2+\om^2}e^{\be t}\left[(1+\be)\sin \om t -\om\cos \om  t\right]\nonumber
\end{eqnarray}
and
\begin{eqnarray}\label{def.M3}
  M_3(t,k) &=& -ik\times \frac{c_1(k)}{\si}e^{\si t}\\
  &&-ik\times\frac{c_2(k)}{\be^2+\om^2}e^{\be t}\left[\be\cos \om t +\om\sin \om  t\right]\nonumber\\
    &&-ik\times \frac{c_3(k)}{\be^2+\om^2}e^{\be t}\left[\be\sin \om t -\om\cos \om  t\right].\nonumber
\end{eqnarray}

Now, let us summarize the above computations on the explicit representation of Fourier transforms of the solution $U=[\rho, u, E,B]$.

\begin{theorem}\label{thm.green}
Let $U=[\rho, u, E,B]$ be the solution to the Cauchy problem \eqref{lhs}-\eqref{lhs.id} on the linearized homogeneous system with initial data $U_0=[\rho_0,u_0,E_0,B_0]$ satisfying  \eqref{lhs.cc}. For $t\geq 0$ and $k\in \R^3$ with $|k|\neq 0$, one has the decomposition
\begin{equation}\label{thm.green.1}
    \left[\begin{array}{c}
    \hat{\rho}(t,k)\\
    \hat{u}(t,k)\\
    \hat{E}(t,k)\\
    \hat{B}(t,k)
    \end{array}\right]=
     \left[\begin{array}{c}
    \hat{\rho}(t,k)\\
    \hat{u}_{\parallel}(t,k)\\
    \hat{E}_{\parallel}(t,k)\\
    0
    \end{array}\right]+
     \left[\begin{array}{c}
    0\\
    \hat{u}_\perp(t,k)\\
    \hat{E}_\perp(t,k)\\
    \hat{B}_\perp(t,k)
    \end{array}\right],
\end{equation}
where $ \hat{u}_{\parallel}, \hat{u}_\perp$ are defined by
\begin{equation*}
   \hat{u}_{\parallel}=\tilde{k}\tilde{k}\cdot\hat{u},\ \ \hat{u}_\perp=-\tilde{k}\times (\tilde{k}\times \hat{u})=(\FI_3-\tilde{k}\otimes \tilde{k})\hat{u},
\end{equation*}
and likewise for $\hat{E}_{\parallel},\hat{E}_\perp$ and $\hat{B}_\perp$. Denote
\begin{equation}\label{thm.green.2}
  \left[\begin{array}{c}
    M_1(t,k)\\
     M_2(t,k)\\
     M_3(t,k)
    \end{array}\right]:=\left[\begin{array}{c}
    \hat{u}_\perp(t,k)\\
    \hat{E}_\perp(t,k)\\
    \hat{B}_\perp(t,k)
    \end{array}\right],\ \
    \left[\begin{array}{c}
    M_{1,0}(k)\\
     M_{2,0}(k)\\
    M_{3,0}(k)
    \end{array}\right]:=\left[\begin{array}{c}
    \hat{u}_{0,\perp}(k)\\
    \hat{E}_{0,\perp}(k)\\
    \hat{B}_{0,\perp}(k)
    \end{array}\right]
\end{equation}
Then, there are matrices $G^{I}_{7\times 7}(t,k)$ and $G^{II}_{9\times 9}(t,k)$ such that
\begin{equation}\label{thm.green.3}
\left[\begin{array}{c}
    \hat{\rho}(t,k)\\
    \hat{u}_{\parallel}(t,k)\\
    \hat{E}_{\parallel}(t,k)
    \end{array}\right]= G^{I}_{7\times 7}(t,k)\left[\begin{array}{c}
    \hat{\rho}_0(k)\\
    \hat{u}_{0,\parallel}(k)\\
    \hat{E}_{0,\parallel}(k)
    \end{array}\right]
\end{equation}
and
\begin{equation*}
 \left[\begin{array}{c}
    M_1(t,k)\\
     M_2(t,k)\\
     M_3(t,k)
    \end{array}\right]= G^{II}_{9\times 9}(t,k)\left[\begin{array}{c}
    M_{1,0}(k)\\
     M_{2,0}(k)\\
    M_{3,0}(k)
    \end{array}\right],
\end{equation*}
where $G^{I}_{7\times 7}(t,k)$ is given by
\begin{eqnarray}
  G^{I}_{7\times 7} &=& e^{-\frac{t}{2}}\cos(\sqrt{3/4+\ga|k|^2} t) \FI_3\label{thm.green.5}\\
  &&+e^{-\frac{t}{2}}\frac{\sin(\sqrt{3/4+\ga|k|^2} t)}{\sqrt{3/4+\ga|k|^2}}
  \left[\begin{array}{ccc}
     1/2 & -i k         & 0 \\
     -i \ga k    & -1/2 & -1  \\
     0           &    1         & 1/2
  \end{array}\right],\nonumber
\end{eqnarray}
and $G^{II}_{9\times 9}(t,k)$ is explicitly determined by representations \eqref{def.M1}, \eqref{def.M2}, \eqref{def.M3} for $M_1(t,k)$, $M_2(t,k)$, $M_3(t,k)$ with $c_1(k)$, $c_2(k)$ and $c_3(k)$ defined by \eqref{def.C123} in terms of $M_{1,0}(k)$, $M_{2,0}(k)$, $M_{3,0}(k)$.
\end{theorem}

\subsection{Refined $L^p$-$L^q$ time-decay property}\label{sec.lhs.4}

In this subsection, we use Theorem \ref{thm.green} to obtain some refined  $L^p$-$L^q$ time-decay property for each component in the solution $[\rho,u,E,B]$. For that, we first find the delicate time-frequency pointwise estimates on the Fourier transforms $\hat{\rho},\hat{u},\hat{E}$ and $\hat{B}$ in the following

\begin{lemma}\label{lem.4tfbd}
Let $U=[\rho,u,E,B]$ be the solution to the linearized homogeneous system \eqref{lhs} with initial data $U_0=[\rho_0,u_0,E_0,B_0]$ satisfying \eqref{lhs.cc}. Then, there are constants $\la>0$, $C>0$ such that for all $t\geq 0, k\in\R^3$,
\begin{equation}\label{lem.4tfbd.1}
    |\hat{\rho}(t,k)|\leq C e^{-\frac{t}{2}}|[\hat{\rho}_0(k),\hat{u}_0(k)]|,
\end{equation}
\begin{multline}
\label{lem.4tfbd.2}
 |\hat{u}(t,k)|\leq   C e^{-\frac{t}{2}}|[\hat{\rho}_0(k),\hat{u}_0(k),\hat{E}_0(k)]|\\[2mm]
 +C|[\hat{u}_0(k),\hat{E}_0(k),\hat{B}_0(k)]|\cdot
\left\{\begin{array}{ll}
   \left(e^{-\la t}+|k|e^{-\la |k|^2 t}\right) &\ \ \text{if}\ |k|\leq 1\\[2mm]
    \left(e^{-\la t}+\frac{1}{|k|}e^{-\frac{\la t}{|k|^2}}\right) &\ \ \text{if}\ |k|\geq 1,
 \end{array}\right.
\end{multline}
\begin{multline}
\label{lem.4tfbd.3}
 |\hat{E}(t,k)|\leq   C e^{-\frac{t}{2}}|[\hat{u}_0(k),\hat{E}_0(k)]|\\[2mm]
 +C|[\hat{u}_0(k),\hat{E}_0(k),\hat{B}_0(k)]|\cdot
\left\{\begin{array}{ll}
   \left(e^{-\la t}+|k|e^{-\la |k|^2 t}\right) &\ \ \text{if}\ |k|\leq 1\\[2mm]
    \left(\frac{1}{|k|^2}e^{-\la t}+e^{-\frac{\la t}{|k|^2}}\right) &\ \ \text{if}\ |k|\geq 1,
 \end{array}\right.
\end{multline}
and
\begin{equation}\label{lem.4tfbd.4}
    |\hat{B}(t,k)|\leq C|[\hat{u}_0(k),\hat{E}_0(k),\hat{B}_0(k)]| \cdot
\left\{\begin{array}{ll}
   \left(|k|e^{-\la t}+e^{-\la |k|^2 t}\right) &\ \ \text{if}\ |k|\leq 1\\[2mm]
    \left(\frac{1}{|k|}e^{-\la t}+e^{-\frac{\la t}{|k|^2}}\right) &\ \ \text{if}\ |k|\geq 1.
    \end{array}\right.
\end{equation}

\end{lemma}

\begin{proof}
Recall the decomposition \eqref{thm.green.1} of $[\hat{\rho},\hat{u},\hat{E},\hat{B}]$. It is straightforward to obtain upper bounds of each component in the first part $[\hat{\rho},\hat{u}_\parallel,\hat{E}_\parallel,0]$ due to \eqref{thm.green.3} and \eqref{thm.green.5}, which lead to \eqref{lem.4tfbd.1} and the first term on the r.h.s. of both \eqref{lem.4tfbd.2} and \eqref{lem.4tfbd.3}. The rest is to find the upper bounds of the second part $[0,\hat{u}_\perp,\hat{E}_\perp,\hat{B}_\perp]$ or equivalently $[M_1,M_2,M_3]$ in terms of $[\hat{u}_0,\hat{E}_0,\hat{B}_0]$ by \eqref{thm.green.2}. Next, let us consider the upper bound of $M_1(t,k)$ defined in \eqref{def.M1}. In fact, by Lemma \ref{a.lem.root}, it is straightforward to check \eqref{def.C123}  to obtain
\begin{align*}
&  \left[\begin{array}{c}
     c_1 \\ c_2 \\ c_3
   \end{array}\right]=
     \left[\begin{array}{rrr}
      -O(1)|k|^2\FI_3 &   -O(1)|k|^2\FI_3    &  O(1)|k|i\tilde{k}\times\\
      -O(1)|k|^2\FI_3 &         O(1)\FI_3    & -O(1)|k|i\tilde{k}\times\\
            O(1)\FI_3 &         O(1)\FI_3    &  O(1)|k|i\tilde{k}\times
    \end{array}\right]
    \left[\begin{array}{c}
       M_{1,0}\\ M_{2,0} \\ M_{3,0}
    \end{array}\right]
\end{align*}
as $|k|\to 0$, and
\begin{align*}
&  \left[\begin{array}{c}
     c_1 \\ c_2 \\ c_3
   \end{array}\right]=
     \left[\begin{array}{rrr}
      -O(1)|k|^{-2}\FI_3 &   -O(1)|k|^{-4}\FI_3    &  O(1)|k|^{-3}i\tilde{k}\times\\
      -O(1)|k|^{-2}\FI_3 &            O(1)\FI_3    & -O(1)|k|^{-3}i\tilde{k}\times\\
       O(1)|k|^{-1}\FI_3 &    O(1)|k|^{-3}\FI_3    &          O(1)i\tilde{k}\times
    \end{array}\right]
    \left[\begin{array}{c}
       M_{1,0}\\ M_{2,0} \\ M_{3,0}
    \end{array}\right]
\end{align*}
as $|k|\to \infty$. Moreover, one has
\begin{equation*}
    \frac{1+\be}{(1+\be)^2+\om^2}=\left\{
    \begin{array}{ll}
      O(1) & \ \ \text{as $|k|\to 0$},\\
      O(1)|k|^{-2} & \ \ \text{as $|k|\to \infty$},
    \end{array}\right.
\end{equation*}
and
\begin{equation*}
    \frac{\om}{(1+\be)^2+\om^2}=\left\{
    \begin{array}{ll}
      O(1) & \ \ \text{as $|k|\to 0$},\\
      O(1)|k|^{-1} & \ \ \text{as $|k|\to \infty$}.
    \end{array}\right.
\end{equation*}
Therefore, after plugging the above computations into  \eqref{def.M1}, it holds that
\begin{align*}
  M_1(t,k) =& -\left(-O(1)|k|^2M_{1,0}-O(1)|k|^2M_{2,0}+O(1)|k|i\tilde{k}\times M_{3,0}\right)\\
  &\qquad\qquad\qquad\qquad\qquad\qquad \cdot O(1) e^{\si(k) t}\\
  & -\left(O(1)|k|^2M_{1,0}+O(1)M_{2,0}-O(1)|k|i\tilde{k}\times M_{3,0}\right)\\
  & \qquad\qquad\qquad\qquad\qquad\qquad \cdot \left(O(1)\cos \om t + O(1)\sin \om t\right)e^{\be(k)  t}\\
  & -\left(O(1)M_{1,0}+O(1)M_{2,0}+O(1)|k|i\tilde{k}\times M_{3,0}\right)\\
  & \qquad\qquad\qquad\qquad\qquad\qquad \cdot \left(O(1)\sin t - O(1)\cos \om t\right)e^{\be(k)  t},
\end{align*}
as $|k|\to 0$, and
\begin{align*}
  M_1(t,k) =& -\left(-O(1)|k|^{-2}M_{1,0}-O(1)|k|^{-4}M_{2,0}+O(1)|k|^{-3}i\tilde{k}\times M_{3,0}\right)\\
  &\qquad\qquad\cdot O(1)|k|^2e^{\si(k) t}\\
  & -\left(O(1)|k|^{-2}M_{1,0}+O(1)M_{2,0}-O(1)|k|^{-3}i\tilde{k}\times M_{3,0}\right)\\
  & \qquad\qquad\cdot \left(O(1)|k|^{-2}\cos \om t + O(1)|k|^{-1}\sin \om t\right)e^{\be(k)  t}\\
  & -\left(O(1)|k|^{-1}M_{1,0}+O(1)|k|^{-3}M_{2,0}+O(1)i\tilde{k}\times M_{3,0}\right)\\
  & \qquad\qquad\cdot \left(O(1)|k|^{-2}\sin\om t - O(1)|k|^{-1}\cos \om t\right)e^{\be(k)  t},
\end{align*}
as $|k|\to \infty$. Notice that due to Lemma \ref{a.lem.root} again, there is $\la>0$ such that
\begin{equation*}
    \left\{\begin{array}{ll}
       \dis \si(k)\geq -\la |k|^2,\ \ \be(k)=-\frac{\si(k)+1}{2}\geq -\la &\ \ \text{over $|k|\leq 1$},\\[3mm]
        \dis \si(k)\geq -\la,\ \ \be(k)=-\frac{\si(k)+1}{2}\geq -\frac{\la t}{|k|^2} &\ \ \text{over $|k|\geq 1$}.
    \end{array}\right.
\end{equation*}
Therefore, it follows that for $|k|\leq 1$,
\begin{equation*}
    |M_1(t,k)|\leq C (e^{-\la t}+|k|e^{-\la |k|^2t})|[M_{1,0},M_{2,0},M_{3,0}]|,
\end{equation*}
and for $|k|\geq 1$,
\begin{equation*}
    |M_1(t,k)|\leq C \left(e^{-\la t}+\frac{1}{|k|}e^{-\frac{\la t} {|k|^2}}\right)|[M_{1,0},M_{2,0},M_{3,0}]|.
\end{equation*}
Furthermore, since $|[M_{1,0},M_{2,0},M_{3,0}]|\leq |[\hat{u}_0(k),\hat{E}_0(k),\hat{B}_0(k)]|$, one has
\begin{equation*}
     |M_1(t,k)|\leq C|[\hat{u}_0(k),\hat{E}_0(k),\hat{B}_0(k)]|\cdot
\left\{\begin{array}{ll}
   \left(e^{-\la t}+|k|e^{-\la |k|^2 t}\right) &\ \ \text{if}\ |k|\leq 1\\[2mm]
    \left(e^{-\la t}+\frac{1}{|k|}e^{-\frac{\la t}{|k|^2}}\right) &\ \ \text{if}\ |k|\geq 1,
 \end{array}\right.
\end{equation*}
that is the upper bound of $\hat{u}_\perp(t,k)$ corresponding to the second term on the r.h.s. of \eqref{lem.4tfbd.2}. Hence, \eqref{lem.4tfbd.2} is proved. Finally, \eqref{lem.4tfbd.3} and \eqref{lem.4tfbd.4} can be proved in the completely same way as for  \eqref{lem.4tfbd.2}. Here, we only mention that to estimate $M_3(t,k)$ defined in \eqref{def.M3}, we need to use
\begin{equation*}
    \frac{\be}{\be^2+\om^2}=\left\{
    \begin{array}{ll}
      -O(1) & \ \ \text{as $|k|\to 0$},\\
      -O(1)|k|^{-4} & \ \ \text{as $|k|\to \infty$},
    \end{array}\right.
\end{equation*}
and
\begin{equation*}
    \frac{\om}{\be^2+\om^2}=\left\{
    \begin{array}{ll}
      O(1) & \ \ \text{as $|k|\to 0$},\\
      O(1)|k|^{-1} & \ \ \text{as $|k|\to \infty$}.
    \end{array}\right.
\end{equation*}
All the rest details are omitted for simplicity. This completes the proof of Lemma \ref{lem.4tfbd}.
\end{proof}

Based on Lemma \ref{lem.4tfbd}, the time-decay property for the full solution $[\rho,u,E,B]$ obtained in Theorem \ref{thm.lhs} can be improved as follows.

\begin{theorem}\label{thm.lhs.re}
Let $1\leq p,r\leq 2\leq q\leq \infty$, $\ell\geq 0$ and let $m\geq 0$ be an integer. Suppose $U(t)=e^{tL}U_0$ is the solution to the Cauchy problem \eqref{lhs}-\eqref{lhs.id} with initial data $U_0=[\rho_0,u_0,E_0,B_0]$ satisfying \eqref{lhs.cc}. Then, $U=[\rho,u,E,B]$ satisfies the following time-decay property:
\begin{equation}\label{thm.lhs.re.1}
\|\na^m \rho(t)\|_{L^q}\leq C e^{-\frac{t}{2}} \left(\|[\rho_0,u_0]\|_{L^p}+\|\na^{m+[3(\frac{1}{r}-\frac{1}{q})]_+}[\rho_0,u_0]\|_{L^r}\right),
\end{equation}
\begin{multline}
\label{thm.lhs.re.2}
\|\na^m u(t)\|_{L^q}
\leq C e^{-\frac{t}{2}}\left(\|\rho_0\|_{L^p}+\|\na^{m+[3(\frac{1}{r}-\frac{1}{q})]_+}\rho_0\|_{L^r}\right)\\
\qquad+C(1+t)^{-\frac{3}{2}(\frac{1}{p}-\frac{1}{q})-\frac{m+1}{2}}\|[u_0,E_0,B_0]\|_{L^p}\\
+C(1+t)^{-\frac{\ell+1}{2}}
\|\na^{m+[\ell+3(\frac{1}{r}-\frac{1}{q})]_+} [u_0,E_0,B_0]\|_{L^r},
\end{multline}
\begin{multline}
\label{thm.lhs.re.3}
\|\na^m E(t)\|_{L^q}
\leq C(1+t)^{-\frac{3}{2}(\frac{1}{p}-\frac{1}{q})-\frac{m+1}{2}}\|[u_0,E_0,B_0]\|_{L^p}\\
+C(1+t)^{-\frac{\ell}{2}}
\|\na^{m+[\ell+3(\frac{1}{r}-\frac{1}{q})]_+} [u_0,E_0,B_0]\|_{L^r},
\end{multline}
and
\begin{multline}
\label{thm.lhs.re.4}
\|\na^m B(t)\|_{L^q}
\leq C(1+t)^{-\frac{3}{2}(\frac{1}{p}-\frac{1}{q})-\frac{m}{2}}\|[u_0,E_0,B_0]\|_{L^p}\\
+C(1+t)^{-\frac{\ell}{2}}
\|\na^{m+[\ell+3(\frac{1}{r}-\frac{1}{q})]_+} [u_0,E_0,B_0]\|_{L^r},
\end{multline}
for any $t\geq 0$, where $C=C(p,q,r,\ell,m)$ and $[\ell+3(\frac{1}{r}-\frac{1}{q})]_+$ is defined in \eqref{def.index}.
\end{theorem}

\begin{proof}
Take $1\leq p,r\leq 2\leq q\leq \infty$ and an integer $m\geq 0$. Similar to \eqref{thm.lhs.p1}, it follows from \eqref{lem.4tfbd.1} that
\begin{equation*}
    \|\na^m\rho(t)\|_{L^q_x}\leq C e^{-\frac{t}{2}}\left(\left\||k|^{m}[\hat{\rho}_0,\hat{u}_0]\right\|_{L^{q'}(|k|\leq 1)}+\left\||k|^{m}[\hat{\rho}_0,\hat{u}_0]\right\|_{L^{q'}(|k|\geq 1)}\right).
\end{equation*}
It further holds that
\begin{equation*}
  \left\||k|^{m}[\hat{\rho}_0,\hat{u}_0]\right\|_{L^{q'}(|k|\leq 1)}\leq C\|[\rho_0,u_0]\|_{L^p}
\end{equation*}
and
\begin{equation*}
  \left\||k|^{m}[\hat{\rho}_0,\hat{u}_0]\right\|_{L^{q'}(|k|\geq 1)}\leq C  \|\na^{m+[3(\frac{1}{r}-\frac{1}{q})]_+}[\rho_0,u_0]\|_{L^r}
\end{equation*}
where we obtained the second inequality by using the similar method as for \eqref{thm.lhs.p4} which can be applied with $\ell=0$. Then, \eqref{thm.lhs.re.1} follows. To prove \eqref{thm.lhs.re.2}, it similarly holds that
\begin{equation*}
    \|\na^m u(t)\|_{L^q_x}\leq C \left\||k|^{m}\hat{u}(t)\right\|_{L^{q'}(|k|\leq 1)}+C\left\||k|^{m}\hat{u}(t)\right\|_{L^{q'}(|k|\geq 1)}.
\end{equation*}
where from \eqref{lem.4tfbd.2}, the first part is bounded by
\begin{multline*}
 \left\||k|^{m}\hat{u}(t)\right\|_{L^{q'}(|k|\leq 1)} \leq C e^{-\frac{t}{2}}\|\rho_0\|_{L^p}+Ce^{-\la t}\|[u_0,E_0,B_0]\|_{L^p}\\
 +C(1+t)^{-\frac{3}{2}(\frac{1}{p}-\frac{1}{q})-\frac{m+1}{2}}\|[u_0,E_0,B_0]\|_{L^p},
\end{multline*}
and the second part is bounded by
\begin{multline*}
 \left\||k|^{m}\hat{u}(t)\right\|_{L^{q'}(|k|\geq 1)} \leq C e^{-\frac{t}{2}}\left\||k|^{m+(\frac{1}{r}-\frac{1}{q})(3+\eps)}
 [\hat{\rho}_0,\hat{u}_0,\hat{E}_0]\right\|_{L^{r'}(|k|\geq 1)}\\
 +Ce^{-\la t}\left\||k|^{m+(\frac{1}{r}-\frac{1}{q})(3+\eps)}
 [\hat{u}_0,\hat{E}_0,\hat{B}_0]\right\|_{L^{r'}(|k|\geq 1)}\\
 +C(1+t)^{-\frac{\ell+1}{2}}\left\||k|^{m+\ell+(\frac{1}{r}-\frac{1}{q})(3+\eps)}
 [\hat{u}_0,\hat{E}_0,\hat{B}_0]\right\|_{L^{r'}(|k|\geq 1)}.
\end{multline*}
Here, $\ell\geq 0$, $\eps>0$ is a small enough constant, and also we used
\begin{equation*}
    \sup_{|k|\geq 1} \left(\frac{1}{|k|^{\ell+1}}e^{-\frac{\la t}{4|k|^2}}\right) \leq C(1+t)^{-\frac{\ell+1}{2}}.
\end{equation*}
Collecting the above estimates on $u$ yields \eqref{thm.lhs.re.2}. In the completely same way, \eqref{thm.lhs.re.3} and \eqref{thm.lhs.re.4} follows from \eqref{lem.4tfbd.3} and \eqref{lem.4tfbd.4}, respectively and details of proof are omitted for simplicity. Here, we only remark that the first term on the r.h.s. of \eqref{lem.4tfbd.3} results from the fact that the term decaying in the slowest time rate over $|k|\leq 1$ on the r.h.s. of \eqref{lem.4tfbd.3} is
\begin{equation*}
C|[\hat{u}_0(k),\hat{E}_0(k),\hat{B}_0(k)]|
\cdot |k|e^{-\la |k|^2 t}
\end{equation*}
and hence
\begin{multline*}
   \left\||k|^{m+1}e^{-\la |k|^2 t}[\hat{u}_0(k),\hat{E}_0(k),\hat{B}_0(k)]\right\|_{L^{q'}(|k|\leq 1)}\\
   \leq   C(1+t)^{-\frac{3}{2}(\frac{1}{p}-\frac{1}{q})-\frac{m+1}{2}}\|[u_0,E_0,B_0]\|_{L^p}.
\end{multline*}
This completes the proof of Theorem \ref{thm.lhs.re}.
\end{proof}

For later use, from Theorem \ref{thm.lhs.re}, let us list some special cases in the following

\begin{corollary}\label{cor.lhs.re}
Suppose $U(t)=e^{tL}U_0$ is the solution to the Cauchy problem \eqref{lhs}-\eqref{lhs.id} with initial data $U_0=[\rho_0,u_0,E_0,B_0]$ satisfying \eqref{lhs.cc}. Then, $U=[\rho,u,E,B]$ satisfies the following time-decay property:
\begin{equation}\label{cor.linear.s.1}
    \left\{\begin{array}{l}
      \dis \|\rho(t)\|\leq Ce^{-\frac{t}{2}} \|[\rho_0,u_0]\|,\\[1mm]
\dis \|u(t)\| \leq  C e^{-\frac{t}{2}}\|\rho_0\|+C (1+t)^{-\frac{5}{4}}\|[u_0,E_0,B_0]\|_{L^1\cap \dot{H}^2},\\[1mm]
\dis \|E(t)\|\leq C(1+t)^{-\frac{5}{4}}\|[u_0,E_0,B_0]\|_{L^1\cap\dot{H}^3},\\[1mm]
\dis   \|B(t)\| \leq  C(1+t)^{-\frac{3}{4}}\|[u_0,E_0,B_0]\|_{L^1\cap\dot{H}^2},
    \end{array}\right.
\end{equation}
and
\begin{equation}\label{cor.linear.s.2}
    \left\{\begin{array}{l}
     \dis\|\rho(t)\|_{L^\infty}\leq Ce^{-\frac{t}{2}} \|[\rho_0,u_0]\|_{L^2\cap \dot{H}^2},\\[1mm]
\dis\|u(t)\|_{L^\infty} \leq  C e^{-\frac{t}{2}}\|\rho_0\|_{L^2\cap \dot{H}^2}+C (1+t)^{-2}\|[u_0,E_0,B_0]\|_{L^1\cap \dot{H}^5},\\[1mm]
\dis\|E(t)\|_{L^\infty}\leq C(1+t)^{-2}\|[u_0,E_0,B_0]\|_{L^1\cap\dot{H}^6},\\[1mm]
\dis  \|B(t)\|_{L^\infty} \leq  C(1+t)^{-\frac{3}{2}}\|[u_0,E_0,B_0]\|_{L^1\cap\dot{H}^5},
    \end{array}\right.
\end{equation}
and moreover,
\begin{equation}\label{cor.linear.s.3}
    \left\{\begin{array}{l}
     \dis \|\na B(t)\|\leq  C(1+t)^{-\frac{5}{4}}\|[u_0,E_0,B_0]\|_{L^1\cap \dot{H}^4},\\[1mm]
\dis \|\na^N[E(t),B(t)]\|\leq C(1+t)^{-\frac{5}{4}}\|[u_0,E_0,B_0]\|_{L^2\cap \dot{H}^{N+3}}.
    \end{array}\right.
\end{equation}

\end{corollary}

\section{Decay in time for the nonlinear system}\label{sec.nonldecay}

In this section, we shall prove Proposition \ref{prop.s.decay} and Proposition \ref{prop.s.Lq} by bootstrap argument. Concerning the solution $V=[\si,v,\widetilde{E},\widetilde{B}]$ to the nonlinear Cauchy problem \eqref{s.e}-\eqref{s.e.id}, the first two subsections are devoted to obtaining the time-decay rates of the full instant energy $\|V(t)\|_N^2$ and the high-order instant energy $\|\na V(t)\|_{N-1}^2$, respectively, and in the last subsection, we investigate the time-decay rates in $L^q$ with $2\leq q\leq \infty$ for each component $\si,v,\widetilde{E}$ and $\widetilde{B}$ of the solution $V$.

In what follows, since we shall apply the linear $L^p$-$L^q$ time-decay property of the homogeneous system \eqref{lhs} studied in the previous section to the nonlinear case,  we need the mild form of the original nonlinear Cauchy problem \eqref{s.cr}-\eqref{s.cr.id}. Throughout this section, we suppose that $U=[\rho,u,E,B]$ is the solution to the Cauchy problem \eqref{s.cr}-\eqref{s.cr.id} with initial data $U_0=[\rho_0,u_0,E_0,B_0]$ satisfying \eqref{s.cr.cc}. Here, we remark that due to the transform \eqref{trans.1}, Proposition \ref{prop.s.exi} also holds for $U$. Then, the solution $U$ can be formally written as
\begin{equation}\label{mildf}
    U(t)=e^{t L} U_0+\int_0^t e^{(t-s)L}[g_1(s),g_2(s),g_3(s),0]ds,
\end{equation}
where $e^{tL}$ is defined in \eqref{def.lu} and the nonlinear source term takes the form of
\begin{equation}\label{mildf.s}
    \left\{\begin{array}{l}
      \dis g_1=-\na\cdot(\rho u),\\
      \dis g_2=-u\cdot \na u -u\times B-\ga[(1+\rho)^{\ga -2}-1]\na \rho,\\
      \dis g_3=\rho u.
    \end{array}\right.
\end{equation}
It should be pointed out that in the time integral term of \eqref{mildf}, given $0\leq s\leq t$, it makes sense that $e^{(t-s)L}$ acts on $[g_1(s),g_2(s),g_3(s),0]$ since $[g_1(s),g_2(s),g_3(s),0]$ satisfies the compatible condition \eqref{lhs.cc}.

\subsection{Time rate for the full instant energy functional}

In this subsection we shall prove the time-decay estimate \eqref{prop.s.decay.1} in Proposition \ref{prop.s.decay} for the full instant energy $\|V(t)\|_N^2$. The starting point is the following lemma which can be seen directly from the proof of Proposition \ref{prop.s.exi}.

\begin{lemma}
Let $V=[\si,v,\widetilde{E},\widetilde{B}]$ be the solution to the Cauchy problem \eqref{s.e}-\eqref{s.e.id} with initial data $V_0=[\si_0,v_0,\widetilde{E}_0,\widetilde{B}_0]$ satisfying \eqref{s.e.cc} in the sense of Proposition \ref{prop.s.exi}. Then, if $\CE_N(V_0)$ is sufficiently small,
\begin{equation}\label{lem.ief.1}
    \frac{d}{dt}\CE_{N}(V(t))+\la \CD_N(V(t))\leq 0
\end{equation}
holds for any $t\geq 0$, where $\CE_{N}(V(t))$, $\CD_N(V(t))$ are in the form of \eqref{def.ef} and \eqref{def.ef.d}, respectively.
\end{lemma}

Notice $\CE_{N}(V(t))
\sim\|V(t)\|_N^2$, and hence it is equivalent to consider their time-decay rates. Though \eqref{lem.ief.1} implies that $\CE_{N}(V(t))$ is a non-increasing in time Lyapunov functional, its dissipation rate $\CD_N(V(t))$ is so weak that it does not include both the zero-term $\|\widetilde{B}(t)\|^2$ and the highest-order term $\|\na^N[\widetilde{E}(t),\widetilde{B}(t)]\|^2$. The main idea of overcoming these two difficulties is that for the latter, we apply the time-weighted estimate to the inequality \eqref{lem.ief.1} and use iteration in both the time rate and the derivative order to remove the regularity-loss effects of the  dissipative rate $\CD_N(V(t))$, and for the former, we apply the linear $L^p$-$L^q$ time-decay to bound $\|\widetilde{B}(t)\|^2$ in terms of initial data and the nonlinear source term. The similar idea has been mentioned in \cite{DUY}.

Now, we begin with the time-weighted estimate and iteration for the Lyapunov inequality \eqref{lem.ief.1}. Let $\ell\geq 0$. Multiplying \eqref{lem.ief.1} by $(1+t)^{\ell}$ and taking integration over $[0,t]$ gives
\begin{multline*}
    (1+t)^{\ell}\CE_{N}(V(t))+\la \int_0^t (1+s)^{\ell}\CD_{N}(V(s))ds\\
    \leq \CE_{N}(V_0)+\ell \int_0^t(1+s)^{\ell-1}\CE_{N}(V(s)).
\end{multline*}
Noticing
\begin{equation*}
\CE(V)\leq C(\|\widetilde{B}\|^2+\CD_{N+1}(V)),
\end{equation*}
it follows that
\begin{multline*}
    (1+t)^{\ell}\CE_{N}(V(t))+\la \int_0^t (1+s)^{\ell}\CD_{N}(V(s))ds\\
    \leq \CE_{N}(V_0)+C\ell \int_0^t(1+s)^{\ell-1}\|\widetilde{B}(s)\|^2ds\\
    + C\ell \int_0^t(1+s)^{\ell-1}\CD_{N+1}(V(s)))ds.
\end{multline*}
Similarly, it holds that
\begin{multline*}
    (1+t)^{\ell-1}\CE_{N+1}(V(t))+\la \int_0^t (1+s)^{\ell-1}\CD_{N+1}(V(s))ds\\
    \leq \CE_{N+1}(V_0)+C(\ell-1) \int_0^t(1+s)^{\ell-2}\|\widetilde{B}(s)\|^2ds\\
    + C(\ell-1) \int_0^t(1+s)^{\ell-2}\CD_{N+2}(V(s)))ds,
\end{multline*}
and
\begin{equation*}
  \CE_{N+2}(V(t))+  \la \int_0^t \CD_{N+2}(V(s))ds\leq  \CE_{N+2}(V_0).
\end{equation*}
Then, for $1<\ell<2$, it follows by iterating the above estimates that
\begin{multline}\label{ief.p1}
    (1+t)^{\ell}\CE_{N}(V(t))+\la \int_0^t (1+s)^{\ell}\CD_{N}(V(s))ds\\
    \leq C\CE_{N+2}(V_0)+C\int_0^t(1+s)^{\ell-1}\|\widetilde{B}(s)\|^2ds.
\end{multline}

On the other hand, to estimate the integral term on the r.h.s. of \eqref{ief.p1}, let us define
\begin{equation}\label{def.e.infty}
    \CE_{N,\infty}(V(t))=\sup_{0\leq s\leq t}(1+s)^{\frac{3}{2}} \CE_{N}(V(s)).
\end{equation}
Then, we have the following

\begin{lemma}\label{lem.B2.bdd}
For any $t\geq 0$, it holds that
\begin{equation}\label{lem.B2.bdd.1}
    \|\widetilde{B}(t)\|^2\leq C(1+t)^{-\frac{3}{2}}\left(\|[v_0,\widetilde{E}_0,\widetilde{B}_0]\|_{L^1\cap \dot{H}^2}^2+[\CE_{N,\infty}(V(t))]^2\right).
\end{equation}
\end{lemma}

\begin{proof}
Apply the fourth linear estimate on $B$ in \eqref{cor.linear.s.1} to the mild form \eqref{mildf} so that
\begin{multline}
\label{lem.B2.bdd.p1}
  \|B(t)\| \leq  C(1+t)^{-\frac{3}{4}}\|[u_0,E_0,B_0]\|_{L^1\cap \dot{H}^2}\\
+C\int_0^t (1+s)^{-\frac{3}{4}} \|[g_2(s),g_3(s)]\|_{L^1\cap \dot{H}^2}ds .
\end{multline}
Recall the definition \eqref{mildf.s} of $g_2$ and $g_3$. It is straightforward to verify that for any $0\leq s\leq t$,
\begin{equation*}
   \|[g_2(s),g_3(s)]\|_{L^1\cap \dot{H}^2} \leq C \CE_N(U(s)).
\end{equation*}
Notice that $\CE_N(U(s))\leq C \CE_N(V(\sqrt{\ga}s))$. From \eqref{def.e.infty}, for any $0\leq s\leq t$,
\begin{equation*}
    \CE_N(V(\sqrt{\ga}s))\leq (1+\sqrt{\ga}s)^{-\frac{3}{2}}\CE_{N,\infty}(V(\sqrt{\ga}t)).
\end{equation*}
Then, it follows that for $0\leq s\leq t$,
\begin{equation*}
   \|[g_2(s),g_3(s)]\|_{L^1\cap \dot{H}^2} \leq C (1+\sqrt{\ga}s)^{-\frac{3}{2}}\CE_{N,\infty}(V(\sqrt{\ga}t)).
\end{equation*}
Putting this into \eqref{lem.B2.bdd.p1} gives
\begin{equation*}
    \|B(t)\|\leq C(1+t)^{-\frac{3}{4}}\big(\|[u_0,E_0,B_0]\|_{L^1\cap \dot{H}^2}+\CE_{N,\infty}(V(\sqrt{\ga}t))\big)
\end{equation*}
which implies \eqref{lem.B2.bdd.1} since $\|\widetilde{B}(t)\|\leq C\|B(t/\sqrt{\ga})\|$ and $[u_0,E_0,B_0]$ is equivalent with $[v_0,\widetilde{E}_0,\widetilde{B}_0]$ up to a positive constant.
This completes the proof of Lemma \ref{lem.B2.bdd}.
\end{proof}

Now, the rest is to prove the uniform-in-time boundedness of $\CE_{N,\infty}(V(t))$ which yields the time-decay rates of the Lyapunov functional $\CE_N(V(t))$ and thus $\|V(t)\|_N^2$. In fact, by taking $\ell=\frac{3}{2}+\eps$ in \eqref{ief.p1} with $\eps>0$ small enough, one has
\begin{multline*}
    (1+t)^{\frac{3}{2}+\eps}\CE_{N}(V(t))+\la \int_0^t (1+s)^{\frac{3}{2}+\eps}\CD_{N}(V(s))ds\\
    \leq C\CE_{N+2}(V_0)+C\int_0^t(1+s)^{\frac{1}{2}+\eps}\|\widetilde{B}(s)\|^2ds.
\end{multline*}
Here, using \eqref{lem.B2.bdd.1} and the fact that $\CE_{N,\infty}(V(t))$ is non-decreasing in $t$, it further holds that
\begin{equation*}
   \int_0^t(1+s)^{\frac{1}{2}+\eps}\|\widetilde{B}(s)\|^2ds
    \leq C(1+t)^\eps \left(\|[v_0,\widetilde{E}_0,\widetilde{B}_0]\|_{L^1\cap \dot{H}^2}^2+[\CE_{N,\infty}(V(t))]^2\right).
\end{equation*}
Therefore, it follows that
\begin{multline*}
 (1+t)^{\frac{3}{2}+\eps}\CE_{N}(V(t))+\la \int_0^t (1+s)^{\frac{3}{2}+\eps}\CD_{N}(V(s))ds\\
    \leq C\CE_{N+2}(V_0) +C(1+t)^\eps \left(\|[v_0,\widetilde{E}_0,\widetilde{B}_0]\|_{L^1\cap \dot{H}^2}^2+[\CE_{N,\infty}(V(t))]^2\right),
\end{multline*}
which implies
\begin{equation*}
    (1+t)^{\frac{3}{2}}\CE_{N}(V(t))\leq  C\left(\CE_{N+2}(V_0)+\|[v_0,\widetilde{E}_0,\widetilde{B}_0]\|_{L^1}^2+[\CE_{N,\infty}(V(t))]^2\right),
\end{equation*}
and thus
\begin{equation*}
    \CE_{N,\infty}(V(t))\leq  C\left(\eps_{N+2}(V_0)^2+[\CE_{N,\infty}(V(t))]^2\right).
\end{equation*}
Here, recall the definition \eqref{def.eps.id} of $\eps_{N+2}(V_0)$. Since $\eps_{N+2}(V_0)>0$ is sufficiently small, $ \CE_{N,\infty}(V(t))\leq  C\eps_{N+2}(V_0)^2$ holds true for any $t\geq 0$, which implies
\begin{equation*}
    \|V(t)\|_N\leq C  \CE_{N}(V(t))^{1/2}\leq C\eps_{N+2}(V_0)(1+t)^{-\frac{3}{4}},
\end{equation*}
that is \eqref{prop.s.decay.1}. This completes the proof of the first part of Proposition \ref{prop.s.decay}.

\subsection{Time rate for the high-order instant energy functional}

In this subsection, we shall continue the proof of Proposition \ref{prop.s.decay} for the second part \eqref{prop.s.decay.2}, that is the time-decay estimate of the high-order energy $\|\na V(t)\|_{N-1}^2$. In fact, it can reduce to the time-decay estimates only on $\|\na \widetilde{B}\|$ and $\|\na^N[\widetilde{E},\widetilde{B}]\|$ by the following lemma.

\begin{lemma}\label{lem.ief.h}
Let $V=[\si,v,\widetilde{E},\widetilde{B}]$ be the solution to the Cauchy problem \eqref{s.e}-\eqref{s.e.id} with initial data $V_0=[\si_0,v_0,\widetilde{E}_0,\widetilde{B}_0]$ satisfying \eqref{s.e.cc} in the sense of Proposition \ref{prop.s.exi}. Then, if $\CE_N(V_0)$ is sufficiently small, there are the high-order instant energy functional $\CE_N^{\rm h}(\cdot)$ and the corresponding dissipation rate $\CD_N^{\rm h}(\cdot)$ such that
\begin{equation}\label{lem.ief.h.1}
    \frac{d}{dt}\CE_{N}^{\rm h}(V(t))+\la \CD_N^{\rm h}(V(t))\leq C\|\na \widetilde{B}\|^2,
\end{equation}
holds for any $t\geq 0$.
\end{lemma}

\begin{proof}
It can be done by modifying the proof of Theorem \ref{thm.ap} a little. In fact, by letting the energy estimates made only on the high-order derivatives, then corresponding to \eqref{thm.ap.p01}, \eqref{thm.ap.p03}, \eqref{thm.ap.p08}, and \eqref{thm.ap.p09}, it can be re-verified that
\begin{equation*}
    \frac{1}{2}\frac{d}{dt}\|\na V\|_{N-1}^2+\frac{1}{\sqrt{\ga}}\|\na v\|_{N-1}^2\leq C\|V\|_N\|\na [\si,v]\|_{N-1}^2,
\end{equation*}
\begin{equation*}
\frac{d}{dt}\sum_{1\leq |\al|\leq N-1}\langle \pa^\al v, \pa^\al \na \si\rangle+\la \|\na \si\|_{N-1}^2\\
\leq C\|\na^2 v\|_{N-2}^2+\|V\|_N^2\|\na[\si,v]\|_{N-1}^2,
\end{equation*}
\begin{multline*}
\frac{d}{dt}\sum_{1\leq |\al|\leq N-1}\langle \pa^\al v, \pa^\al\widetilde{E}\rangle+\la \|\na\widetilde{E}\|_{N-2}^2\\
\leq C\|\na v\|_{N-1}^2+C\|\na^2 v\|_{N-2}\|\na \widetilde{B}\|_{N-2}+\|V\|_N^2\|\na[\si,v]\|_{N-1}^2,
\end{multline*}
and
\begin{multline*}
\frac{d}{dt}\sum_{1\leq |\al|\leq N-2}\langle \na\times \pa^\al \widetilde{E}, \pa^\al\widetilde{B}\rangle+\la \|\na^2\widetilde{B}\|_{N-3}^2\\
\leq C\|\na^2\widetilde{E}\|_{N-3}^2+C\|\na v\|_{N-3}^2+\|V\|_N^2\|\na[\si,v]\|_{N-1}^2.
\end{multline*}
Here, the details of proof are omitted for simplicity. Now, in the similar way as in \eqref{def.energy}, let us define
\begin{multline}\label{def.energy.h}
\CE_{N}^{\rm h}(V(t))   = \|\na V\|_{N-1}^2+\kappa_1\sum_{1\leq |\al|\leq N-1}\langle \pa^\al v, \pa^\al \na \si\rangle\\
+\kappa_2\sum_{1\leq |\al|\leq N-1}\langle \pa^\al v, \pa^\al\widetilde{E}\rangle
+\kappa_3\sum_{1\leq |\al|\leq N-2}\langle \na\times \pa^\al \widetilde{E}, \pa^\al\widetilde{B}\rangle.
\end{multline}
Similarly, one can choose $0<\kappa_3\ll \kappa_2\ll\kappa_1\ll 1$ with $\kappa_2^{3/2}\ll \kappa_3$ such that $\CE_{N}^{\rm h}(V(t))\sim \|\na V(t)\|_{N-1}^2$, that is, $\CE_{N}^{\rm h}(\cdot)$ is indeed a high-order instant energy functional satisfying \eqref{def.ef.h}, and furthermore, the linear combination of the previously obtained four estimates with coefficients corresponding to \eqref{def.energy.h} yields \eqref{lem.ief.h.1} with $ \CD_N^{\rm h}(\cdot)$ defined in \eqref{def.ef.h.d}. This completes the proof of Lemma \ref{lem.ief.h}.
\end{proof}

By comparing \eqref{def.ef.h} and \eqref{def.ef.h.d} for the definitions of $ \CE_N^{\rm h}(\cdot)$ and $ \CD_N^{\rm h}(\cdot)$, it follows from \eqref{lem.ief.h.1} that
\begin{equation*}
    \frac{d}{dt}\CE_{N}^{\rm h}(V(t))+\la \CE_{N}^{\rm h}(V(t))\leq C(\|\na \widetilde{B}\|^2+\|\na^N[\widetilde{E},\widetilde{B}]\|^2),
\end{equation*}
which implies
\begin{multline}\label{decay.hef.p1}
\CE_{N}^{\rm h}(V(t))\leq \CE_{N}^{\rm h}(V_0)e^{-\la t}\\
+C\int_0^t e^{-\la (t-s)}(\|\na \widetilde{B}(s)\|^2+\|\na^N[\widetilde{E}(s),\widetilde{B}(s)]\|^2)ds.
\end{multline}
To estimate the time integral term on the r.h.s. of the above inequality, one has

\begin{lemma}\label{lem.hef.source}
Let $V=[\si,v,\widetilde{E},\widetilde{B}]$ be the solution to the Cauchy problem \eqref{s.e}-\eqref{s.e.id} with initial data $V_0=[\si_0,v_0,\widetilde{E}_0,\widetilde{B}_0]$ satisfying \eqref{s.e.cc} in the sense of Proposition \ref{prop.s.exi}. If $\eps_{N+6}(V_0)>0$ is sufficiently small, where $\eps_{N+6}(V_0)$ is defined in \eqref{def.eps.id}, then
\begin{equation}\label{lem.hef.source.1}
    \|\na \widetilde{B}(t)\|^2+\|\na^N[\widetilde{E}(t),\widetilde{B}(t)]\|^2\leq C\eps_{N+6}(V_0)^2(1+t)^{-\frac{5}{2}}
\end{equation}
for any $t\geq 0$.
\end{lemma}

\noindent For this time, suppose that the above lemma is true. Then, by using \eqref{lem.hef.source.1} in \eqref{decay.hef.p1}, one has
\begin{equation*}
 \CE_{N}^{\rm h}(V(t))\leq    \CE_{N}^{\rm h}(V_0)e^{-\la t}+C\eps_{N+6}(V_0)^2(1+t)^{-\frac{5}{2}}.
\end{equation*}
Since $ \CE_{N}^{\rm h}(V(t))\sim \|\na V(t)\|_{N-1}^2$ holds true for any $t\geq 0$, \eqref{prop.s.decay.2} follows. This also completes the proof of Proposition \ref{prop.s.decay}.
The rest is devoted to

\medskip

\noindent{\bf Proof of Lemma \ref{lem.hef.source}:}
Suppose that  $\eps_{N+6}(V_0)>0$ is sufficiently small. Notice that, by the first part of Proposition \ref{prop.s.decay},
\begin{equation*}
    \|V(t)\|_{N+4}\leq  C\eps_{N+6}(V_0)(1+t)^{-\frac{3}{4}},
\end{equation*}
which further implies from \eqref{trans.1} that for $U=[\rho,u,E,B]$,
\begin{equation}\label{lem.hef.source.p1}
    \|U(t)\|_{N+4}\leq  C\eps_{N+6}(V_0)(1+t)^{-\frac{3}{4}}.
\end{equation}
Similarly to obtain \eqref{lem.B2.bdd.p1}, one can apply the linear estimate \eqref{cor.linear.s.3} to the mild form \eqref{mildf} of the solution $U(t)$ so that
\begin{multline}
\label{lem.hef.source.p2}
  \|\na B(t)\| \leq  C(1+t)^{-\frac{5}{4}}\|[u_0,E_0,B_0]\|_{L^1\cap \dot{H}^4}\\
+C\int_0^t (1+s)^{-\frac{5}{4}} \|[g_2(s),g_3(s)]\|_{L^1\cap \dot{H}^4}ds,
\end{multline}
and
\begin{multline}
\label{lem.hef.source.p3}
  \|\na^N [E(t),B(t)]\| \leq  C(1+t)^{-\frac{5}{4}}\|[u_0,E_0,B_0]\|_{L^2\cap \dot{H}^{N+3}}\\
+C\int_0^t (1+s)^{-\frac{5}{4}} \|[g_2(s),g_3(s)]\|_{L^2\cap \dot{H}^{N+3}}ds.
\end{multline}
Recalling the definition \eqref{mildf.s} of $g_2$ and $g_3$, it is straightforward to check that
\begin{eqnarray*}
&&\|[g_2(t),g_3(t)]\|_{L^1\cap \dot{H}^4}\leq C\|U(t)\|_{\max\{5,N\}}^2,\\
&& \|[g_2(t),g_3(t)]\|_{L^2\cap \dot{H}^{N+3}}\leq C\|U(t)\|_{N+4}^2.
\end{eqnarray*}
The above estimate together with \eqref{lem.hef.source.p1} give
\begin{equation*}
   \|[g_2(t),g_3(t)]\|_{L^1\cap \dot{H}^4}+ \|[g_2(t),g_3(t)]\|_{L^2\cap \dot{H}^{N+3}}\leq C\eps_{N+6}(V_0)^2(1+t)^{-\frac{3}{2}}.
\end{equation*}
Then, it follows from \eqref{lem.hef.source.p2} and \eqref{lem.hef.source.p3} that
\begin{equation*}
  \|\na B(t)\|+  \|\na^N [E(t),B(t)]\|\leq    C\eps_{N+6}(V_0)(1+t)^{-\frac{5}{4}},
\end{equation*}
where the smallness of $\eps_{N+6}(V_0)$ was used. This implies \eqref{lem.hef.source.1} by the definition \eqref{trans.1} of $\widetilde{E}$ and $\widetilde{B}$. The proof of Lemma \ref{lem.hef.source} is complete.

\subsection{Time rate in $L^q$}

In this subsection we shall prove Proposition \ref{prop.s.decay} for the time-decay rates of solutions $V=[\si,v,\widetilde{E},\widetilde{B}]$ in $L^q$ with $2\leq q\leq \infty$ to the Cauchy problem \eqref{s.e}-\eqref{s.e.id}. To prove \eqref{prop.s.Lq.1}, \eqref{prop.s.Lq.2} and \eqref{prop.s.Lq.3}, due to Proposition \ref{prop.s.exi} and the transform \eqref{trans.1},  it equivalently suffices to consider the same estimates on $U=[\rho,u,E,B]$ which is the solution to the other reformulated Cauchy problem \eqref{s.cr}-\eqref{s.cr.id}. Throughout this subsection, we suppose that $\eps_{13}(V_0)>0$ is sufficiently small. In addition, for $N\geq 4$, Proposition \ref{prop.s.decay} shows that if $\eps_{N+2}(V_0)>0$ is sufficiently small,
\begin{eqnarray}
\|U(t)\|_{N}\leq C\eps_{N+2}(V_0)(1+t)^{-\frac{3}{4}},\label{rateq.p1}
\end{eqnarray}
and if $\eps_{N+6}(V_0)>0$ is sufficiently small,
\begin{eqnarray}\label{rateq.p2}
\|\na U(t)\|_{N-1}\leq C\eps_{N+6}(V_0)(1+t)^{-\frac{5}{4}}.
\end{eqnarray}
Now, we begin with estimates on $B$, $[u,E]$ and $\rho$ in turn as follows.

\medskip

\noindent{\it Estimate on $\|B\|_{L^q}$.}  For $L^2$ rate, it is easy to see from \eqref{rateq.p1} that
\begin{equation*}
    \|B(t)\|\leq C\eps_{6}(V_0)(1+t)^{-\frac{3}{4}}.
\end{equation*}
For $L^\infty$ rate, by applying the $L^\infty$ linear estimate on $B$ in \eqref{cor.linear.s.2} to the mild form \eqref{mildf}, one has
\begin{multline*}
\|B(t)\|_{L^\infty}\leq C(1+t)^{-\frac{3}{2}}\|[u_0,E_0,B_0]\|_{L^1\cap \dot{H}^5}\\
+C\int_0^t(1+t-s)^{-\frac{3}{2}} \|[g_2(s),g_3(s)]\|_{L^1\cap \dot{H}^5}ds.
\end{multline*}
Since by \eqref{rateq.p1},
\begin{equation*}
 \|[g_2(t),g_3(t)]\|_{L^1\cap \dot{H}^5}\leq C\|U(t)\|_6^2 \leq C\eps_{8}(V_0)^2(1+t)^{-\frac{3}{2}},
\end{equation*}
it follows that
\begin{equation*}
 \|B(t)\|_{L^\infty}\leq C\eps_{8}(V_0)(1+t)^{-\frac{3}{2}}.
\end{equation*}
So, by $L^2$-$L^\infty$ interpolation,
\begin{equation}\label{rateq.p.rho}
    \|B(t)\|_{L^q}\leq C\eps_{8}(V_0)(1+t)^{-\frac{3}{2}+\frac{3}{2q}}
\end{equation}
for $2\leq q\leq \infty$.

\medskip

\noindent{\it Estimate on $\|[u,E]\|_{L^q}$.} For $L^2$ rate, applying the $L^2$ linear estimates on $u$ and $E$ in \eqref{cor.linear.s.1} to \eqref{mildf}, one has
\begin{multline*}
\|u(t)\|\leq C(1+t)^{-\frac{5}{4}}(\|\rho_0\|+\|[u_0,E_0,B_0]\|_{L^1\cap \dot{H}^2})\\
+C\int_0^t(1+t-s)^{-\frac{5}{4}}(\|g_1(s)\|+ \|[g_2(s),g_3(s)]\|_{L^1\cap \dot{H}^2})ds,
\end{multline*}
and
\begin{multline*}
\|E(t)\|\leq C(1+t)^{-\frac{5}{4}}\|[u_0,E_0,B_0]\|_{L^1\cap \dot{H}^3}\\
+C\int_0^t(1+t-s)^{-\frac{5}{4}} \|[g_2(s),g_3(s)]\|_{L^1\cap \dot{H}^3}ds.
\end{multline*}
Since by \eqref{rateq.p1},
\begin{equation*}
\|g_1(t)\|+\|[g_2(t),g_3(t)]\|_{L^1\cap {H}^3}\leq C\|U(t)\|_4^2 \leq C\eps_{6}(V_0)^2(1+t)^{-\frac{3}{2}},
\end{equation*}
it follows that
\begin{equation*}
    \|[u(t),E(t)]\|\leq C\eps_{6}(V_0)(1+t)^{-\frac{5}{4}}.
\end{equation*}
For $L^\infty$ rate, applying the $L^\infty$ linear estimates on $u$ and $E$ in \eqref{cor.linear.s.2} to \eqref{mildf}, one has
\begin{multline*}
\|u(t)\|_{L^\infty}\leq C(1+t)^{-2}(\|\rho_0\|_{L^2\cap\dot{H}^2}+\|[u_0,E_0,B_0]\|_{L^1\cap \dot{H}^5})\\
+C\int_0^t(1+t-s)^{-2}(\|g_1(s)\|_{L^2\cap\dot{H}^2}+ \|[g_2(s),g_3(s)]\|_{L^1\cap \dot{H}^5})ds,
\end{multline*}
and
\begin{multline*}
\|E(t)\|_{L^\infty}\leq C(1+t)^{-2}\|[u_0,E_0,B_0]\|_{L^1\cap \dot{H}^6}\\
+C\int_0^t(1+t-s)^{-2}\|[g_2(s),g_3(s)]\|_{L^1\cap \dot{H}^6}ds,
\end{multline*}
Since
\begin{equation*}
\|g_1(t)\|_{L^2\cap \dot{H}^2}+ \|[g_2(t),g_3(t)]\|_{\dot{H}^5\cap \dot{H}^6}
\leq C\|\na U(t)\|_6^2\leq  C\eps_{13}(V_0)^2(1+t)^{-\frac{5}{2}}
\end{equation*}
and
\begin{multline*}
 \|[g_2(t),g_3(t)]\|_{L^1}\leq C\|U(t)\|(\|u(t)\|+\|\na U(t)\|)\\
 \leq C  \left[\eps_{6}(V_0)(1+t)^{-\frac{3}{4}}\right]\cdot \left[\eps_{10}(V_0)(1+t)^{-\frac{5}{4}}\right]
 \leq C \eps_{10}(V_0)^2(1+t)^{-2}
\end{multline*}
where  \eqref{rateq.p1}, \eqref{rateq.p2} and \eqref{rateq.p3} were used, then it follows that
\begin{equation*}
    \|[u(t),E(t)]\|_{L^\infty}\leq C\eps_{13}(V_0)(1+t)^{-2}.
\end{equation*}
Therefore, by $L^2$-$L^\infty$ interpolation,
\begin{equation}\label{rateq.p3}
    \|[u(t),E(t)]\|_{L^q}\leq C\eps_{13}(V_0)(1+t)^{-2+\frac{3}{2q}}
\end{equation}
for $2\leq q\leq \infty$.

\medskip

\noindent{\it Estimate on $\|\rho\|_{L^q}$.} For $L^2$ rate, we need to bootstrap once. First, applying the $L^2$ linear estimates on $\rho$ in \eqref{cor.linear.s.1} to \eqref{mildf}, one has
\begin{equation}\label{rateq.p4}
\|\rho(t)\|\leq Ce^{-\frac{t}{2}}\|[\rho_0,u_0]\|
+C\int_0^te^{-\frac{t-s}{2}}\|[g_1(s),g_2(s)]\|ds.
\end{equation}
Due to
\begin{equation*}
  \|[g_1(t),g_2(t)] \|\leq C(\|\na U(t)\|_1^2+\|u(t)\|\cdot\|B(t)\|_{L^\infty})\leq C\eps_{10}(V_0)^2(1+t)^{-\frac{5}{2}}
\end{equation*}
where  \eqref{rateq.p2},  \eqref{rateq.p.rho} and \eqref{rateq.p3} were used, then \eqref{rateq.p4} gives the slower time-decay estimate
\begin{equation*}
  \|\rho(t)\|\leq C \eps_{10}(V_0)(1+t)^{-\frac{5}{2}}.
\end{equation*}
By further re-estimating $ \|[g_1,g_2] \|$ and using  \eqref{rateq.p2},  \eqref{rateq.p.rho}, \eqref{rateq.p3} once again and the above slower time-decay estimate to obtain
\begin{multline*}
  \|[g_1(t),g_2(t)] \|\leq C\|u(t)\|_{L^\infty}(\|\na \rho(t)\|+\|\na u(t)\|+\|B(t)\|)\\
  +C\|\rho(t)\|
  (\|\na \rho(t)\|_2+\|\na u(t)\|_2)
  \leq C\eps_{13}(V_0)^2(1+t)^{-\frac{11}{4}},
\end{multline*}
it follows from \eqref{rateq.p4} that
\begin{equation*}
  \|\rho(t)\|\leq C \eps_{13}(V_0)(1+t)^{-\frac{11}{4}}.
\end{equation*}
For $L^\infty$ rate,  by applying the $L^\infty$ linear estimates on $\rho$ in \eqref{cor.linear.s.2} to \eqref{mildf},
\begin{equation}\label{rateq.p5}
\|\rho(t)\|_{L^\infty}\leq Ce^{-\frac{t}{2}}\|[\rho_0,u_0]\|_{L^2\cap \dot{H}^2}
+C\int_0^te^{-\frac{t-s}{2}}\|[g_1(s),g_2(s)]\|_{L^2\cap \dot{H}^2}ds.
\end{equation}
Notice that one can check
\begin{multline}\label{rateq.p6}
\|[g_1(t),g_2(t)]\|_{ \dot{H}^2}\leq C\|\na U(t)\|_4(\|\rho(t)\|+\|[u(t),B(t)]\|_{L^\infty}\\
+\|\na[\rho(t),u(t)]\|_{L^\infty}).
\end{multline}
Here, since the linear time-decay rate of $\|\na[\rho(t),u(t)]\|_{L^\infty}$ is larger than $3/2$ and the nonhomogeneous source is at least quadratically nonlinear, we have the following slower time-decay estimate
\begin{equation*}
 \|\na[\rho(t),u(t)]\|_{L^\infty}\leq C  \eps_{8}(V_0)(1+t)^{-\frac{3}{2}}.
\end{equation*}
Then, it follows from \eqref{rateq.p6} that
\begin{equation*}
\|[g_1(t),g_2(t)]\|_{L^2\cap \dot{H}^2}\leq C \eps_{13}(V_0)(1+t)^{-\frac{11}{4}},
\end{equation*}
which implies from \eqref{rateq.p5} that
\begin{equation*}
\|\rho(t)\|_{L^\infty}\leq C \eps_{13}(V_0)(1+t)^{-\frac{11}{4}}.
\end{equation*}
Therefore, by $L^2$-$L^\infty$ interpolation,
\begin{equation}\label{rateq.p7}
\|\rho(t)\|_{L^q}\leq C \eps_{13}(V_0)(1+t)^{-\frac{11}{4}}
\end{equation}
for $2\leq q\leq \infty$.

\medskip

Thus, \eqref{rateq.p7}, \eqref{rateq.p3} and \eqref{rateq.p.rho} give \eqref{prop.s.Lq.1}, \eqref{prop.s.Lq.2} and \eqref{prop.s.Lq.3}, respectively. This completes the proof of Proposition \ref{prop.s.Lq} and hence Theorem \ref{thm.s.o}.

\section*{Acknowledgments}

This work is supported by the Direct Grant 2010/2011 in CUHK. The author also acknowledges the financial support from RICAM, Austrian Academy of Sciences when this work was done there in the early of 2010. The author would like to thank Professor Shuichi Kawashima for sending to him on this September some recent work \cite{IHK,IK,UWK}  about the investigation of the system with regularity-loss property.

\end{document}